\documentclass[11pt]{article}

\usepackage[utf8]{inputenc}
\usepackage[british]{babel}
\usepackage[a4paper, margin=2.5cm]{geometry}
\usepackage{amsmath, amsthm, amssymb, amsfonts}
\usepackage{mathtools}
\usepackage{mathabx}
\usepackage{thmtools}
\usepackage[shortlabels,inline]{enumitem}
\usepackage[font={small, it}]{caption}
\usepackage{subcaption}
\usepackage{hyperref}
\usepackage{xcolor}
\usepackage{bm}
\usepackage{changepage}
\usepackage{microtype}
\usepackage{stmaryrd}

\newcommand{\cP}{\ensuremath{\mathcal P}}

\newcommand{\cV}{\ensuremath{\mathcal V}}
\newcommand{\cW}{\ensuremath{\mathcal W}}
\newcommand{\cX}{\ensuremath{\mathcal X}}

\newcommand{\eps}{\varepsilon}
\renewcommand{\phi}{\varphi}
\renewcommand{\rho}{\varrho}

\DeclareMathOperator*{\N}{\mathbb{N}}

\DeclareMathOperator*{\R}{\mathbb{R}}

\DeclarePairedDelimiter\ceil{\lceil}{\rceil}
\DeclarePairedDelimiter\floor{\lfloor}{\rfloor}

\let\setminus=\smallsetminus

\newcommand{\Gnp}{G_{n, p}}

\newcommand{\mba}{\ensuremath{\mathbf{a}}}
\newcommand{\mbb}{\ensuremath{\mathbf{b}}}

\newcommand{\mbs}{\ensuremath{\mathbf{s}}}
\newcommand{\mbt}{\ensuremath{\mathbf{t}}}
\newcommand{\mbv}{\ensuremath{\mathbf{v}}}

\newcommand{\mbK}{\ensuremath{\mathbf{K}}}
\newcommand{\mbN}{\ensuremath{\mathbf{N}}}

\newcommand{\mbS}{\ensuremath{\mathbf{S}}}
\newcommand{\mbT}{\ensuremath{\mathbf{T}}}
\newcommand{\mbV}{\ensuremath{\mathbf{V}}}
\newcommand{\mbW}{\ensuremath{\mathbf{W}}}
\newcommand{\mbX}{\ensuremath{\mathbf{X}}}

\newcommand{\rev}{\ensuremath{\mathrm{rev}}}

\declaretheorem[parent=section]{theorem}
\declaretheorem[sibling=theorem]{lemma}
\declaretheorem[sibling=theorem]{proposition}

\declaretheorem[sibling=theorem]{corollary}
\declaretheorem[sibling=theorem,style=definition]{definition}

\setlist{itemsep=0.1em, topsep=0.1em, parsep=0.1em, partopsep=0.1em}

\hypersetup{
  colorlinks,
  linkcolor={red!60!black},
  citecolor={green!50!black},
  urlcolor={blue!80!black}
}

\colorlet{RoyalRed}{red!70!black}
\definecolor{RoyalBlue}{rgb}{0.25, 0.41, 0.88}
\definecolor{RoyalAzure}{rgb}{0.0, 0.22, 0.66}

\newlength{\bibitemsep}
\newlength{\bibparskip}
\let\oldthebibliography\thebibliography
\renewcommand\thebibliography[1]{%
  \oldthebibliography{#1}%
  \setlength{\parskip}{\bibitemsep}%
  \setlength{\itemsep}{\bibparskip}%
}

\newcommand{\srbsquare}{\rotatebox{45}{\tiny{\ensuremath{\blacksquare}}}}
\newcommand{\asscheck}{\hskip.2em \null \hfill \srbsquare}

\newenvironment{assumptions}[1][]%
  {\small
    \begin{adjustwidth*}{1em}{1em}
      {\bf Verifying the assumptions of #1.}%
  }%
  {\asscheck\end{adjustwidth*}}

\title{Sprinkling a few random edges doubles the power}

\author{
  Rajko Nenadov\thanks{Department of Mathematics, ETH Z\"{u}rich, 8092
  Z\"{u}rich, Switzerland. Email: \texttt{rnenadov@math.ethz.ch}. Research
  supported in part by SNSF grant 200021-175573.}
  \and
  Milo\v{s} Truji\'{c}\thanks{Institute of Theoretical Computer Science, ETH
  Z\"{u}rich, 8092 Z\"{u}rich, Switzerland. Email: \texttt{mtrujic@inf.ethz.ch}.
  Research supported by grant no.\ 200021 169242 of the Swiss National Science
  Foundation.}
}
\date{}

\begin{document}
\maketitle

\begin{abstract}
  A seminal result by Koml\'os, Sark\"ozy, and Szemer\'edi states that if a
  graph $G$ with $n$ vertices has minimum degree at least $kn/(k + 1)$, for some
  $k \in \N$ and $n$ sufficiently large, then it contains the $k$-th power of a
  Hamilton cycle. This is easily seen to be the largest power of a Hamilton
  cycle one can guarantee, given such a minimum degree assumption. Following a
  recent trend of studying effects of adding random edges to a dense graph, the
  model known as the randomly perturbed graph, Dudek, Reiher, Ruci\'nski, and
  Schacht showed that if the minimum degree is at least $kn/(k + 1) + \alpha n$,
  for any constant $\alpha > 0$, then adding $O(n)$ random edges on top almost
  surely results in a graph which contains the $(k + 1)$-st power of a Hamilton
  cycle. We show that the effect of these random edges is significantly
  stronger, namely that one can almost surely find the $(2k + 1)$-st power. This
  is the largest power one can guarantee in such a setting.
\end{abstract}

\section{Introduction}

A thoroughly studied topic in extremal combinatorics are the \emph{Dirac-type}
questions: for a graph $G$ on $n$ vertices and a monotone property $\cP$, what
is the minimum $\alpha > 0$ such that $\delta(G) \geq \alpha n$ ensures that $G$
possesses $\cP$? A prime example being (hence the name) Dirac's
theorem~\cite{dirac1952some}, stating that every graph on $n$ vertices with
minimum degree at least $n/2$ contains a Hamilton cycle. Given a graph $H$ and
$k \in \N$, the \emph{$k$-th power $H^k$} of $H$ is defined as a graph on the
same vertex set and $\{v, w\} \in E(H^k)$ if and only if $v$ and $w$ are at
distance at most $k$ in $H$. In 1962 P\'osa conjectured~\cite{erdos1964problem}
that if a graph $G$ on $n$ vertices has minimum degree at least $2n/3$, then $G$
contains $C_n^2$---the square of a cycle of length $n$. In other words, $G$
contains the square of a Hamilton cycle. Seeking for a simpler proof of the
Hajnal-Szmer\'{e}di theorem on clique-factors~\cite{hajnal1970proof},
Seymour~\cite{seymour1973problem} had generalised this by conjecturing that
$\delta(G) \geq kn/(k + 1)$ suffices for the $k$-th power of a Hamilton cycle,
for any $k \in \N$. The conjecture was confirmed twenty-odd years later by
Koml\'{o}s, S\'{a}rk\"{o}zy, and Szemer\'{e}di, utilising the regularity lemma
and the blow-up lemma.

\begin{theorem}[Koml\'{o}s, S\'{a}rk\"{o}zy, Szemer\'{e}di~\cite{komlos1998proof}]
  \label{thm:KSS}
  For $k \in \N$, there exists $n_0 \in \N$ such that if $G$ has order $n$ with
  $n \geq n_0$ and $\delta(G) \geq \frac{k}{k + 1}n$, then $G$ contains the
  $k$-th power of a Hamilton cycle.
\end{theorem}

It is known that Theorem~\ref{thm:KSS} is tight with respect to the minimum
degree requirement. However, Dudek, Reiher, Ruci\'{n}ski, and
Schacht~\cite{dudek2020powers} showed that every graph $G$ with such a minimum
degree not only contains the $k$-th power of a Hamilton cycle, but is very close
to containing the $(k + 1)$-st power in the following sense: for every $\alpha >
0$, if the minimum degree of $G$ is $(k/(k + 1) + \alpha)n$, then additionally
sprinkling $O(n)$ random edges on top of it almost surely results in a graph
which contains the $(k + 1)$-st power of a Hamilton cycle. As $O(n)$ random
edges typically form only very few triangles (or more generally, any short
cycles) and no larger cliques, their contribution towards such a structure is
quite limited. This in turn reveals that the original dense graph was already
close to containing $C_n^{k + 1}$ and it only needed a bit of patching here and
there. We improve their result by showing that such a graph $G$ is actually
close to containing $C_n^{2k + 1}$. That is, we show that in the same setting
one can almost surely find the $(2k + 1)$-st power of a Hamilton cycle.

\begin{theorem}\label{thm:main-result}
  For every $k \in \N$ and every $\alpha > 0$ there exists a positive constant
  $C(\alpha, k)$ such that every graph $\Gamma$ with $n$ vertices and
  $\delta(\Gamma) \geq (\frac{k}{k + 1} + \alpha)n$ is such that for $p = C/n$
  w.h.p.\footnote{An event is said to hold with high probability (w.h.p.\ for
  brevity) if the probability that it holds tends to $1$ as $n$ grows.}\ the
  graph $\Gamma \cup \Gnp$\footnote{$\Gnp$ stands for the probability space of
  all graphs with $n$ vertices where each edge exists independently of other
  edges with probability $p$.} contains the $(2k + 1)$-st power of a Hamilton
  cycle.
\end{theorem}

Theorem~\ref{thm:main-result} is asymptotically optimal in every aspect.
Firstly, having $p = o(1/n)$ does not even guarantee a copy of $C_n^{k + 1}$ as
observed by Dudek, Reiher, Ruci\'{n}ski, and Schacht~\cite{dudek2020powers}.

Secondly, one cannot hope to obtain a higher power than the $(2k + 1)$-st. We
demonstrate this for $k = 2$. Take $\alpha = 0.01$ and consider the vertex
partition $V(\Gamma) = X \cup Y$ with $|X| = (1/3 - \alpha)n$ and $|Y| = (2/3 +
\alpha)n$, where $\Gamma[X, Y]$ is a complete bipartite graph, $\Gamma[Y]$ a
complete graph, and $\Gamma[X]$ an empty graph. We aim to show that $\Gamma \cup
\Gnp$, for $p = O(1/n)$, typically does not contain $C_n^{2k + 2} = C_n^6$. Note
that $C_n^6$ contains $\floor{n/7}$ vertex-disjoint copies of $K_7$. As w.h.p.\
$\Gnp$ does not contain $K_4$'s and only has $o(n)$ triangles, all but at most
$o(n)$ copies of $K_7$ in $\Gamma \cup \Gnp$ have to intersect $Y$ in at least
five vertices. A simple calculation shows that every family of vertex-disjoint
$K_7$'s leaves a linear fraction of vertices in $X$ uncovered.

Lastly, having minimum degree $\delta(\Gamma) \geq nk/(k + 1)$ is not enough for
the $(2k + 1)$-st power even for $p = 0.1 \log n/n$. Let $\Gamma$ be a complete
$(k + 1)$-partite graph $V_1 \cup \dotsb \cup V_{k + 1}$ with all colour classes
being of the same size. Since w.h.p.\ $\Gnp$ contains at least $\eps n$ isolated
vertices (where $\eps > 0$ only depends on the chosen constant $0.1$ in $p$)
there is $i \in [k + 1]$ such that at least $\eps n/(k + 1)$ vertices of $V_i$
are isolated in $\Gnp$. Denote the set of these vertices by $I$. As before, in
order for $\Gamma \cup \Gnp$ to contain $C_n^{2k + 1}$ we need a family of
vertex-disjoint $K_{2k + 2}$'s which cover all the vertices. Note that every
copy of $K_{2k + 2}$ which contains a vertex of $I$ necessarily needs to
intersect some $V_j$, $j \neq i$, in at least three vertices, which must form a
clique in $\Gnp$. Moreover, no two vertices in $I$ can be covered by the same
copy of $K_{2k + 2}$. Consequently, $\Gnp$ needs to contain at least $|I| = \eps
n/(k + 1)$ triangles, which w.h.p.\ is not the case.

We suspect that $\delta(\Gamma) \geq kn/(k + 1)$ and $p = C\log n/n$, for a
sufficiently large constant $C > 0$, is enough for $\Gamma \cup \Gnp$ to w.h.p.\
contain $C_n^{2k + 1}$. We leave this as an open problem.

\subsection{Notation}

Given $n \in \N$, we abbreviate $\{1, \dotsc, n\}$ by $[n]$. Similarly, for $n,
m \in \N$ with $n \leq m$, we write $[n, m]$ to denote the set $\{n, n + 1,
\dotsc, m - 1, m\}$. For $a, b, x, y \in \R$ we let $x = (a \pm b)y$ stand for
$(a - b)y \leq x \leq (a + b)y$. We omit the floor and ceiling symbols whenever
they are not crucial. We write $C_{3.6}$ to indicate that the constant $C_{3.6}$
is given by Theorem/Lemma/Proposition 3.6.

Our graph theoretic notation is mostly standard and follows the one
from~\cite{bondy2008graph}. We outline several possibly non-standard usages. For
a graph $G = (V, E)$ and two subsets of vertices $X, Y \subseteq V(G)$, we let
$N_G(X, Y)$ stand for the common neighbourhood of vertices from $X$ into $Y$,
that is $N_G(X, Y) := \bigcap_{x \in X} N_G(x) \cap Y$. We denote by $e_G(X, Y)$
the number of edges with one endpoint in $X$ and the other in $Y$ and by $d_G(X,
Y)$ the density of the bipartite subgraph induced by $X$ and $Y$, namely $d_G(X,
Y) := e_G(X, Y)/(|X||Y|)$. We use $\deg_G(v, X)$ as a shorthand for $|N_G(\{v\},
X)|$. We drop the subscript $G$ whenever it is clear from the context which
graph we are concerned with.

We use upper-case bold letters to denote ordered tuples of sets, e.g.\ $\mbV =
(V_1, \dotsc, V_k)$, and lower-case bold letters to denote ordered tuples of
vertices, e.g.\ $\mbv = (v_1, \dotsc, v_k)$. We sometimes refer to $\mbv$
interchangeably as a tuple and a set of its elements, the usage should be clear
from the context. We write $\rev(\mbV)$ for the tuple obtained by reversing the
order of $\mbV$, that is $\rev(\mbV) = (V_k, \dotsc, V_1)$, and similarly
$\rev(\mbv)$. Given an integer $i \in \N$, we let $\mbV^i$ stand for the $i$-th
element of the tuple $\mbV$. Additionally, $\mbV^{\leq i}$ and $\mbV^{\geq i}$
denote the tuples obtained by considering only the first $i$ and the last $k - i
+ 1$ elements of a $k$-element tuple $\mbV$, respectively. The tuple $(\mbV,
\mbW)$ is obtained by concatenation of $\mbV = (V_1, \dotsc, V_k)$ and $\mbW =
(W_1, \dotsc, W_\ell)$, that is $(\mbV, \mbW) = (V_1, \dotsc, V_k, W_1, \dotsc,
W_\ell)$. Lastly, given a set $X$, we let $\mbV \setminus X$ denote the tuple
obtained by removing the set $X$ from every element of the tuple $\mbV$.

Given $\ell, k \in \N$, $P_\ell^{2k + 1}$ is a graph defined on the vertex set
$\{v_1, \dotsc, v_\ell\}$ with the edge set $\{v_i, v_j\}$ for all distinct $i,
j \in [\ell]$ with $|i - j| \leq 2k + 1$. We refer to it as the \emph{$(2k +
1)$-st power of a path} ($(2k + 1)$-path for short) of size $\ell$. A $(2k +
1)$-cycle is defined as a natural analogue. The ordered sets $(v_1, \dotsc,
v_{2k + 2})$ and $(v_{\ell - 2k + 1}, \dotsc, v_{\ell})$ are called the {\em
endpoints} of the $(2k + 1)$-path and are necessarily copies of $K_{2k + 2}$
($(2k + 2)$-clique for short). We say that a $(2k + 1)$-path \emph{connects} two
$(2k + 2)$-cliques $\mbs$ and $\mbt$, if $\mbs$ and $\mbt$ are its endpoints. A
path $P \subseteq P_\ell^{2k + 1}$ is the \emph{skeleton} of $P_\ell^{2k + 1}$
if $P^{2k + 1} = P_\ell^{2k + 1}$. Lastly, observe that a union of a $(2k +
1)$-path $P$ connecting some $\mbs$ to $\mbt$ and a $(2k + 1)$-path $Q$
connecting $\mbt$ to some $\mbt'$, and which are otherwise vertex-disjoint, is a
$(2k + 1)$-path $P \cup Q$ connecting $\mbs$ to $\mbt'$.

Lastly, since a lot of tedious work goes into checking whether the assumptions
of some lemma are fulfilled for an application, in an attempt to improve
readability we do this in dedicated paragraphs ending with $\srbsquare$.

\section{Outline of the proof}\label{sec:outline}

Similarly to many recent problems of embedding spanning structures into (random)
graphs, we make use of the so-called \emph{absorbing method}. The method was
first introduced by R\"{o}dl, Ruci\'{n}ski, and
Szemer\'{e}di~\cite{rodl2006dirac} (implicitly used before
in~\cite{erdHos1991vertex, krivelevich1997triangle}).

\begin{definition}\label{def:absorber}
  Let $G$ be a graph and $k \in \N$. A $(2k + 1)$-path $P \subseteq G$ is said
  to be \emph{$X$-absorbing} for a set $X \subseteq V(G) \setminus V(P)$ if for
  every $X^\star \subseteq X$ there is a $(2k + 1)$-path $P^\star$ with the same
  endpoints as $P$ and such that $V(P^\star) = V(P) \cup X^\star$.
\end{definition}

On a high level the proof consists of two steps:
\begin{enumerate*}[label=(\textit{\roman*})]
  \item find an $X$-absorbing path $P$ with endpoints $\mba$ and $\mbb$ for some
    large set of vertices $X$;
  \item construct a long $(2k + 1)$-path from $\mbb$ to $\mba$ which contains
    all the vertices of $V(G) \setminus (X \cup V(P))$, no vertex from $P$
    (other than the endpoints), and possibly some vertices of $X$.
\end{enumerate*}
The absorbing property of $P$ allows us to transform this $(2k + 1)$-cycle into
a $(2k + 1)$-cycle which contains all vertices of $G$.

Most of the proofs implementing this strategy rely on some form of a `connecting
lemma'. In our case, such a lemma would say that for arbitrary two $(2k +
2)$-cliques $\mbs$ and $\mbt$, there is a short $(2k + 1)$-path connecting
$\mbs$ to $\mbt$. Unfortunately, one cannot hope for such a statement for the
following reason. In order to `grow' a $(2k + 1)$-path starting at, say, $\mbs =
(v_1, \dotsc, v_{2k + 2})$, the vertices $(v_2, \dotsc, v_{2k + 2})$ need to
have a common neighbour. Since the minimum degree of $\Gamma$ is only $(k/(k +
1) + \alpha)n$, this can easily not be the case even after adding the random
edges on top. Note that if we were looking for the $(k + 1)$-st power of a
cycle, this would not cause troubles as every set of $k + 1$ vertices has a
large common neighbourhood already in the graph $\Gamma$.

In order to go around this issue we need to have much better control over the
endpoints of every constructed $(2k + 1)$-path throughout our embedding
procedure, that is over $(2k + 2)$-cliques to which we apply the `connecting
lemma'. We achieve this by embedding everything carefully into partition classes
given by Szemer\'{e}di's regularity lemma applied to the dense graph $\Gamma$.
Let $V_0 \cup V_1 \cup \dotsb \cup V_t$ be the partition given by the regularity
lemma. A well-known fact is that the reduced graph $R$ inherits the minimum
degree of the graph $\Gamma$ and thus, by Theorem~\ref{thm:KSS}, contains a
spanning $k$-cycle. Moreover, since the minimum degree of $\Gamma$ is
\emph{slightly larger} than $kn/(k + 1)$, every $(k + 1)$-tuple of vertices in
$R$ has several common neighbours. In particular, the partition classes $V_1,
\dotsc, V_{k + 1}, V_z$, for some $z \in \bigcap_{i \in [k + 1]} N_R(i)$, are
all pairwise $\eps$-regular with positive density. From properties of regularity
we have that almost every set of $2k + 2$ vertices from $V_1 \cup \dotsb \cup
V_{k + 1}$ has a significant common neighbourhood inside of the set $V_z$.
Embedding a $(2k + 2)$-clique $(v_1, \dotsc, v_{2k + 2})$ carefully into $(V_1,
\dotsc, V_{k + 1})$, in a way such that $v_{2i - 1}, v_{2i} \in V_i$ for all $i
\in [k + 1]$, allows us to `grow' a $(2k + 1)$-path starting from it: taking any
edge $\{v_{2k + 3}, v_{2k + 4}\}$ given by $\Gnp$ in $\bigcap_{i \in [2k + 2]}
N_\Gamma(v_i) \cap V_z$, extends the $(2k + 2)$-clique into a $(2k + 1)$-path
with $2k + 4$ vertices. Taking an edge $\{v_{2k + 5}, v_{2k + 6}\} \in V_1$,
which is again given by $\Gnp$ and where $v_{2k + 5}, v_{2k + 6}$ lie in the
common neighbourhood of appropriate vertices in $\Gamma$, we extend the path
even further. Utilising the Counting lemma and Janson's inequality, we may
extend it into an arbitrarily large path (of constant size), or connect two $(2k
+ 2)$-cliques which are embedded in a `nice' way (see
Definition~\ref{def:extendible-tuples} for a more formal description of
`niceness').

Having this in mind, the proof goes along the following lines. Initially, take a
random subset $X_i \subseteq V_i$ for all $i \in [t]$ and denote the union of
all these vertices by $X$. By choosing the size of $X_i$'s carefully, every
vertex $x \in X$ still has, say, $\deg_\Gamma(x, V(\Gamma) \setminus X) \geq
(k/(k + 1) + \alpha/2)n$. We use this fact in order to construct a $(2k +
1)$-path $P$ which is $X$-absorbing and uses roughly the same number of vertices
in the remaining partition classes $V_i \setminus X_i$. Next, most of the
vertices remaining in $V_i \setminus X_i$ are to be covered by a long $(2k +
1)$-path $P'$ which is constructed as outlined above: greedily find short $(2k +
1)$-paths whose endpoints can be extended and use this property in order to
merge them into a long path. Furthermore, such a path $P'$ covers all but a
negligible linear fraction of vertices in $V_i \setminus X_i$, for all $i \in
[t]$ and contains the absorbing path $P$ as a subgraph.

As is usual with proofs involving the regularity lemma, the set of `garbage'
vertices $V_0$ is completely out of our control and is not involved in any of
the embedded paths thus far. Let $V_0'$ denote the union of $V_0$ with all the
vertices of $V_i \setminus X_i$ which are not covered by $P'$, for all $i \in
[t]$. Since $V_0'$ is still very small and $X$ is a uniformly at random chosen
set of suitable size, w.h.p.\ all vertices $v \in V_0'$ have $\deg_\Gamma(v, X)
\geq (k/(k + 1) + \alpha/2)|X|$, say. Then we can, analogously as before, find a
$V_0'$-absorbing $(2k + 1)$-path $P''$ which lies inside of the set $X$ in its
entirety. By using the vertices of $X$ and $V_0'$ only and the `extendible'
property of the endpoints of the paths $P'$ and $P''$ we merge them into a $(2k
+ 1)$-cycle. Finally, we use the absorbing property of $P$ and $P''$ to
incorporate all the remaining vertices of $X$ and $V_0'$ into a spanning $(2k +
1)$-cycle.

\section{Szemer\'{e}di's regularity lemma and random graphs}

Let $G$ be a graph and $\eps$ a positive constant. We say that a pair $(V_1,
V_2)$ of disjoint subsets of $V(G)$ is \emph{$\eps$-regular} if for every $U_i
\subseteq V_i$ (for $i \in \{1, 2\}$) of size $|U_i| \ge \eps |V_i|$ we have
\[
  \big| d(V_1, V_2) - d(U_1, U_2) \big| \leq \eps.
\]
In other words, every two sufficiently large subsets of $V_1$ and $V_2$ induce a
bipartite graph which has roughly the same density as the one induced by $V_1$
and $V_2$. A $k$-tuple $(V_1, \dotsc, V_k)$ is said to be \emph{$\eps$-regular}
with density at least $d$ if $(V_i, V_j)$ forms an $\eps$-regular pair with
$d(V_i, V_j) \geq d$ for all $1 \leq i < j \leq k$. For brevity, we sometimes
write $(\eps, d)$-regular to mean $\eps$-regular with density at least $d$. As a
direct consequence of the definition we get the following proposition.

\begin{proposition}\label{prop:large-reg-inheritance}
  Let $0 < \eps < 1/2$ and $(V_1, V_2)$ be an $\eps$-regular pair with
  density $d \in (0, 1)$. If $\eps < \delta \leq 1/2$, then two subsets $U_1
  \subseteq V_1$ and $U_2 \subseteq V_2$ of size $|U_1| \geq \delta|V_1|$ and
  $|U_2| \geq \delta|V_2|$ form an $\eps/\delta$-regular pair with density at
  least $d - \eps$.
\end{proposition}

A remarkable result of Szemer\'edi~\cite{szemeredi1975regular} states that every
graph can be almost completely decomposed into a few $\eps$-regular pairs. We
use the following variant of Szemer\'{e}di's theorem.

\begin{theorem}[Degree form of the regularity lemma~\cite{komlos1996szemeredi}]
  \label{thm:degree-form-regularity}
  For every $\eps > 0$ and $m \geq 1$, there exists an $M(\eps, m) \geq m$ such
  that for every $d \in [0, 1)$ and every graph $G$ with at least $M$ vertices
  the following holds. There exists $t \in [m, M]$, a partition $(V_i)_{i =
  0}^{t}$ of $V(G)$, and a spanning subgraph $G' \subseteq G$ satisfying:
  \begin{enumerate}[label=(\textit{\roman*})]
    \item\label{reg-except} $|V_0| \leq \eps n$,
    \item\label{reg-equipart} $|V_1| = \dotsb = |V_t| \in [(1 - \eps)n/t, n/t]$,
    \item\label{reg-degree} $\deg_{G'}(v) \geq \deg_G(v) - (d + \eps)n$, for all
      $v \in V(G)$,
    \item\label{reg-empty} $e(G'[V_i]) = 0$, for all $i \in [t]$, and
    \item\label{reg-density} for all $1 \leq i < j \leq t$, the pair $(V_i,
      V_j)$ is $\eps$-regular with density either $0$ or at least $d$.
  \end{enumerate}
\end{theorem}

One usually refers to the partition given by
Theorem~\ref{thm:degree-form-regularity} as an \emph{$\eps$-regular partition}
with \emph{exceptional set} $V_0$. Given a graph $G$, a partition $\cV =
(V_i)_{i = 0}^{t}$ of $V(G)$, and a parameter $d \in (0, 1)$, we define the
\emph{$(\eps, d, G, \cV)$-reduced graph} $R$ as a graph on the vertex set $[t]$
with $\{i, j\}$ being an edge of $R$ if and only if $(V_i, V_j)$ is
\emph{$(\eps, d, G)$-regular}, that is forms an $\eps$-regular pair in $G$ with
density at least $d$.

We also make use of the \emph{counting lemma}, a result often accompanying the
regularity lemma (see, e.g.~\cite{rodl2010regularity}).

\begin{lemma}[Counting Lemma]\label{lemma:counting}
  For every graph $H$ and every $\gamma > 0$, there exists $\eps > 0$ such that
  the following holds. Let $\Gamma$ be a graph, $\cW$ a family of $v(H)$
  disjoint subsets of $V(\Gamma)$, and $\sigma \colon V(H) \to \cW$ a bijection
  such that for every $\{v,w\} \in E(H)$ the pair $(\sigma(v), \sigma(w))$ is
  $\eps$-regular. Then the number of embeddings $\phi \colon H \to \Gamma$ such
  that $\phi(v) \in \sigma(v)$ for every $v \in V(H)$ is
  \[
    \Big( \prod_{W \in \cW} |W| \Big) \Big( \prod_{\{v, w\} \in E(H)} \big(
    d(\sigma(v),\sigma(w)) \pm \gamma \big) \Big).
  \]
\end{lemma}

Before we present the main lemma of this section we introduce several
definitions.

\begin{definition}[Bicanonical paths]
  Let $G$ be a graph and $k, \ell \in \N$ such that $\ell \geq k + 1$. For an
  $\ell$-tuple $\mbV = (V_1, \dotsc, V_\ell)$ of (not necessarily disjoint)
  subsets of $V(G)$, we say that a $(2k + 1)$-path $P = (v_1, \dotsc,
  v_{2\ell})$ in $G$ is $\mbV$-\emph{bicanonical} if $v_{2i - 1}, v_{2i} \in
  V_i$, for all $i \in [\ell]$.
\end{definition}

The next definition is the main notion of our proof strategy. Throughout, it is
of uttermost importance that all the $(2k + 1)$-paths we construct have
endpoints which are \emph{extendible}. Namely, we require each of the endpoints
to be such that their vertices have `large' common neighbourhoods into carefully
chosen sets. This enables us to further extend such paths, connect them with
other paths, and finally close a $(2k + 1)$-cycle (see
Figure~\ref{fig:extendible}).

\begin{definition}[Extendible tuples]\label{def:extendible-tuples}
  Let $G$ be a graph, $k \in \N$, $\rho \in (0, 1)$, and let $\mbV = (V_1,
  \dotsc, V_{k + 1})$ be a $(k + 1)$-tuple of subsets of $V(G)$. We say that a
  $(2k + 2)$-tuple $\mbv = (v_1, \dotsc, v_{2(k + 1)})$ of vertices of $G$ is
  \emph{$(\mbV, \rho)$-extendible} if
  \[
    \big| N_G(\mbv^{\geq 2i}, V_i) \big| \geq \rho|V_i|,
  \]
  for every $i \in [k + 1]$.
\end{definition}

\begin{figure}[!htbp]
  \centering
  \includegraphics[scale=1]{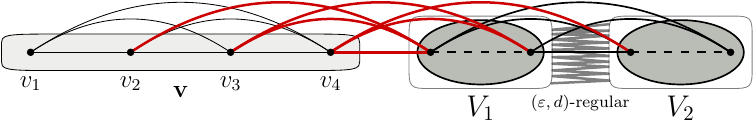}
  \caption{An example of a $(2k + 2)$-tuple $\mbv$ which is $(\mbV,
  \rho)$-extendible for $k = 1$. The grey blobs represent sets of common
  neighbours of vertices in $\mbv$. We depict an example of how to use the
  `extendible' property in order to extend a $4$-clique into a $3$-path of size
  $8$, given that $(V_1, V_2)$ is an $(\eps, d)$-regular pair.}
  \label{fig:extendible}
\end{figure}

\begin{lemma}\label{lem:path-lemma}
  For all $k, \ell \in \N$ and $\mu, d \in (0, 1)$, where $\ell \geq k + 1$,
  there exist positive constants $\eps(d, k, \ell)$, $\rho(d, k,
  \ell)$\footnote{It is crucial for the proof of Theorem~\ref{thm:main-result}
  that $\eps$ and $\rho$ do not depend on $\mu$.}, and $C(\mu, d, k, \ell)$ such
  that for every graph $\Gamma$ with $n$ vertices, the graph $G = \Gamma \cup
  \Gnp$ w.h.p.\ satisfies the following, provided that $p \geq C/n$.

  Let $V_1, \dotsc, V_\ell \subseteq V(G)$ be (not necessarily disjoint) subsets
  such that $|V_i| \geq \mu n$ for each $i \in [\ell]$, and $(V_i, V_j)$ is
  $(\eps, d, \Gamma)$-regular for all distinct $i, j \in [\ell]$, where $|i - j|
  \leq k$. Then there exists a $(V_1, \dotsc, V_\ell)$-bicanonical $(2k +
  1)$-path $P$ in $G$. Moreover, given additional subsets $Y_s, Y_t \subseteq
  V(G)$ of size $|Y_s|, |Y_t| \geq \mu n$ and such that:
  \begin{itemize}
    \item $(Y_s, V_i)$ is $(\eps, d, \Gamma)$-regular for every $i \in [k + 1]$,
      and
    \item $(Y_t, V_i)$ is $(\eps, d, \Gamma)$-regular for every $i \in [\ell -
      k, \ell]$,
  \end{itemize}
  one can find such a $(2k + 1)$-path $P$ connecting some $(2k + 2)$-cliques
  $\mbs$ and $\mbt$ with the following properties:
  \begin{enumerate}[label=(\textit{\roman*})]
    \item\label{s-extendible} $\rev(\mbs)$ is $(\mbS, \rho)$-extendible, where
      $\mbS = (Y_s, V_{k + 1}, \dotsc, V_2) \setminus V(P)$, and
    \item\label{t-extendible} $\mbt$ is $(\mbT, \rho)$-extendible, where $\mbT =
      (Y_t, V_{\ell - k}, \dotsc, V_{\ell - 1}) \setminus V(P)$.
  \end{enumerate}
\end{lemma}

\begin{figure}[!htbp]
  \centering
  \includegraphics[scale=0.75]{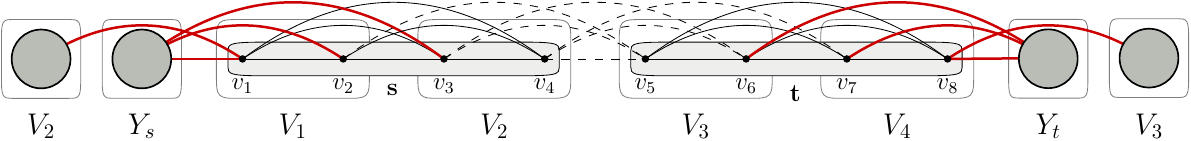}
  \caption{An example of a $(2k + 1)$-path $P_{2\ell}^{2k + 1}$ for $\ell = 4$
  and $k = 1$ given by Lemma~\ref{lem:path-lemma}. The grey blobs represent sets
  of common neighbours of respective vertices of $\mbs$ and $\mbt$ where the
  neighbourhood relations are given by thick red lines.}
  \label{fig:path-lemma}
\end{figure}

\begin{proof}
  We aim to show that a desired $(2k + 1)$-path exists with probability at least
  $1 - 2^{-(\ell + 4)n}$ for any particular (fixed) choice of subsets $V_1,
  \dotsc, V_\ell, Y_s, Y_t$ satisfying the stated requirements. If this is
  indeed the case, then the proof can be easily completed by the union bound.
  Without loss of generality we can assume that $Y_s$ and $Y_t$ are given, as
  otherwise we can artificially add two such sets of vertices to $\Gamma$ and
  connect them to all other vertices. There are at most $n$ choices for the size
  of each subset and at most $2^n$ choices for each subset, thus the probability
  that a desired $(2k + 1)$-path does not exist for at least one valid choice of
  subsets is at most $2^{(\ell + 3)n - (\ell + 4) n} = o(1)$. For the rest of
  the proof we consider one such valid choice of $V_1, \dotsc, V_\ell, Y_s, Y_t
  \subseteq V(G)$.

  Let $P_{2\ell}^{2k + 1}$ denote a $(2k + 1)$-path of size $2\ell$ with the
  vertex set $\mbv = (v_1, \dotsc, v_{2\ell})$. For each $i \in [k + 1]$, let
  $P_i \subseteq P_{2\ell}^{2k + 1}$ be a path (ordinary path) given by the
  following sequence of vertices (see Figure~\ref{fig:skeletons}):
  \[
    \mbv_i = (v_{2i - 1}, v_{2i}, v_{2(i + (k + 1)) - 1}, v_{2(i + (k + 1))},
    \dotsc, v_{2(i + j(k + 1)) - 1}, v_{2(i + j(k + 1))}, \dots),
  \]
  where $0 \le j \le \floor{(\ell - i)/(k + 1)}$, and set $P^\star = \bigcup_{i
  \in [k + 1]} P_i$. Note that $P^\star$ is simply a collection of $k + 1$
  vertex-disjoint paths.

  \begin{figure}[!htbp]
    \centering
    \includegraphics[scale=0.6]{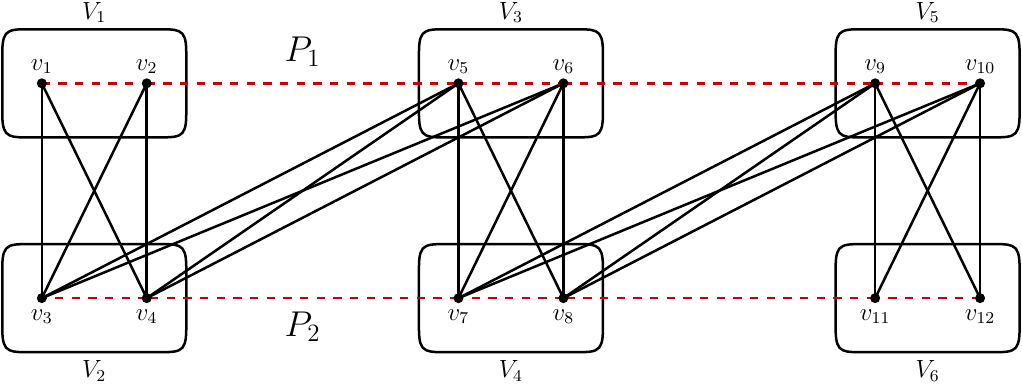}
    \caption{An example of a $(2k + 1)$-path $P_{2\ell}^{2k + 1}$ for $\ell = 6$
    and $k = 1$. The paths $P_i$ are given by dashed red lines and come from
    $\Gnp$. The edges of $R$ are given by thick solid lines and come from
    $\Gamma$.}
    \label{fig:skeletons}
  \end{figure}

  We aim to find a desired $(V_1, \dotsc, V_\ell)$-bicanonical $(2k + 1)$-path
  $P \subseteq G$ such that the edges belonging to a copy of $R = P_{2\ell}^{2k
  + 1} \setminus E(P^\star)$ come from $\Gamma$, and the edges belonging to
  copies of the paths $P_1, \dotsc, P_{k + 1}$ come from $\Gnp$ (see
  Figure~\ref{fig:skeletons}). This is achieved in two steps: in the first step,
  using the Counting Lemma (Lemma~\ref{lemma:counting}), we show that there are
  many embeddings $\phi \colon V(R) \to V(\Gamma)$ such that $\phi(v_{2i - 1}),
  \phi(v_{2i}) \in V_i$ for each $i \in [\ell]$, and \ref{s-extendible} and
  \ref{t-extendible} are satisfied for naturally chosen $\mbs = (\phi(v_1),
  \dotsc, \phi(v_{2(k + 1)}))$ and $\mbt = (\phi(v_{2(\ell - k) - 1}), \dotsc,
  \phi(v_{2\ell}))$; in the second step we apply (a corollary of) Janson's
  inequality to conclude that with probability at least $1 - 2^{-(\ell + 4)n}$
  one of these copies of $R$ is completed to a $(2k + 1)$-path using edges from
  $\Gnp$. As every such $(2k + 1)$-path satisfies \ref{s-extendible} and
  \ref{t-extendible}, this concludes the proof.

  To this end, choose pairwise disjoint subsets
  \begin{itemize}
    \item $W_{2i - 1}$, $W_{2i} \subseteq V_i$ of size $|V_i|/(4\ell)$ for all
      $i \in [\ell]$,
    \item $Y_s' \subseteq Y_s$ of size $|Y_s|/(4\ell)$ and $Y_t' \subseteq Y_t$
      of size $|Y_t|/(4\ell)$,
    \item $A_i \subseteq V_i$ of size $|A_i| = |V_i|/(4\ell)$, for all $i \in
      [2, k + 1]$, and
    \item $B_i \subseteq V_i$ of size $|B_i| = |V_i|/(4\ell)$, for all $i \in
      [\ell - k, \ell - 1]$.
  \end{itemize}
  Let $\cW$ be the family of all the obtained subsets. Finally, let $R^+$ be an
  auxiliary graph obtained by adding the vertices $y_s$, $y_t$, $a_i$ for $i \in
  [2, k + 1]$, and $b_i$ for $i \in [\ell - k, \ell - 1]$ to $R$, together with
  the following edges (see Figure~\ref{fig:graph-r-plus}):
  \begin{itemize}
    \item $\{y_s, v_i\}$ for $i \in [2(k + 1) - 1]$, and $\{y_t, v_i\}$ for $i
      \in [2(\ell - k), 2\ell]$,
    \item $\{a_i, v_j\}$ for $i \in [2, k + 1]$ and $j \in [2(i - 1) - 1]$, and
    \item $\{b_i, v_j\}$ for $i \in [\ell - k, \ell - 1]$ and $j \in [2(i + 1),
      2\ell]$.
  \end{itemize}

  \begin{figure}[!htbp]
    \centering
    \includegraphics[scale=0.8]{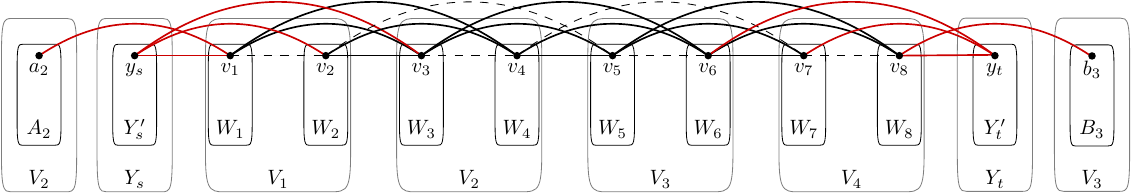}
    \caption{An example of a graph $R^+$ for $\ell = 4$ and $k = 1$. The thick
    solid lines represent the edges of $R^+$ where the red lines represent the
    edges incident to newly added vertices. The dashed lines in a copy of $R$ in
    $G$ come from $\Gnp$ and are depicted just to put things into the right
    perspective.}
    \label{fig:graph-r-plus}
  \end{figure}

  It is important to notice that there are no edges among newly added vertices.
  With this at hand we apply the Counting Lemma (Lemma~\ref{lemma:counting})
  with $R^+$ (as $H$), some $\gamma$ to be specified later, and $\sigma \colon
  V(R^+) \to \cW$ defined as
  \[
    \sigma(v_i) = W_i, \quad \sigma(a_i) = A_i, \quad \sigma(b_i) = B_i, \quad
    \sigma(y_s) = Y_s', \quad \sigma(y_t) = Y_t'.
  \]
  We briefly justify that we indeed may do so. Note that for every $\{v, w\} \in
  E(R^+)$ we have $\sigma(v) \subseteq V'$ and $|\sigma(v)| = |V'|/(4\ell)$, and
  $\sigma(w) \subseteq W'$ and $|\sigma(w)| = |W'|/(4\ell)$, for some disjoint
  subsets $V'$ and $W'$ which are $(\eps, d, \Gamma)$-regular. It follows from
  Proposition~\ref{prop:large-reg-inheritance} that $(\sigma(v), \sigma(w))$ is
  a $4\ell\eps$-regular pair with density at least $d - \eps \geq d/2$, for
  every $\{v, w\} \in E(R^+)$ and hence $\sigma$ satisfies the condition of
  Lemma~\ref{lemma:counting} for sufficiently small $\eps = \eps(\gamma, R^+)$.
  Therefore, there are at least
  \[
    \Big( \prod_{1 \leq i \leq 2\ell} |W_i| \Big) \Big( \prod_{2 \leq i \leq k+1} |A_i|
    \Big) \Big( \prod_{\ell-k \leq i \leq \ell-1} |B_i| \Big) \cdot |Y_s'||Y_t'|
    \cdot \Big(\frac{d}{2} - \gamma \Big)^{e(R^+)} \geq \Big( \prod_{W \in \cW}
    |W| \Big) \Big( \frac{d}{4} \Big)^{e(R^+)}
  \]
  copies of $R^+$ in $\Gamma$ which respect $\sigma$, for $\gamma = d/4$. By
  counting the number of extensions, a simple averaging argument shows that, for
  $\rho = \frac{1}{8\ell} (\frac{1}{4\ell})^{2k + 2} (\frac{d}{4})^{e(R^+)}$ and
  some $\zeta = \zeta(\mu, \rho, d, \ell)$, there are at least $\zeta n^{2\ell}$
  embeddings $\phi \colon V(R) \to V(\Gamma)$ which respect
  $\sigma_{\restriction V(R)}$ and, moreover, $\mbs = (\phi(v_1), \dotsc,
  \phi(v_{2(k + 1)}))$ satisfies \ref{s-extendible} and $\mbt = (\phi(v_{2(\ell
  - k) - 1}), \dotsc, \phi(v_{2\ell}))$ satisfies \ref{t-extendible}. Let us
  denote the family of all such embeddings of $R$ with $\Phi$. This finishes the
  first step of the proof.

  We now show that with probability at least $1 - 2^{-(\ell + 4)n}$ there exists
  $\phi \in \Phi$ such that $\phi(\mbv)$ forms a $(2k + 1)$-path in $\Gamma \cup
  \Gnp$. For each $\phi \in \Phi$, let $P_\phi \subseteq K_n$ be a graph formed
  by $k + 1$ disjoint paths given by $\phi(\mbv_i)$ for $i \in [k + 1]$. In
  other words, $P_\phi$ is a copy of $P^\star$ in $K_n$ given by $\phi(\mbv)$.
  Thus, if $P_\phi \subseteq \Gnp$ then $\phi(\mbv)$ forms a desired $(2k +
  1)$-path. By a corollary of Janson's
  inequality~\cite[Theorem~3.2]{dudek2020powers} we have
  \[
    \Pr[P_\phi \not\subseteq \Gnp \text{ for all } \phi \in \Phi] \le 2^{-c
    \cdot \zeta^2 n^2p},
  \]
  for some constant $c = c(P^\star)$, which is clearly at most $2^{-(\ell +
  4)n}$ for sufficiently large constant $C$.
\end{proof}

Almost immediately we get the following corollary which allows us to connect
several pairs of $(2k + 2)$-cliques by disjoint $(2k + 1)$-paths going through
some regular tuples. Note that, in contrast to Lemma~\ref{lem:path-lemma}, the
cliques $\mbs_i$ and $\mbt_i$ are given as a part of the input and are used in
an `opposite way'; that is we have that $\mbs_i$'s and $\rev(\mbt_i)$'s are
extendible and we want to construct a $(2k + 1)$-path that connects them. In
particular, we need not worry about whether $\rev(\mbs_i)$'s and $\mbt_i$'s are
extendible in order to connect them with other structures.

\begin{corollary}\label{cor:connecting-lemma}
  For all $k \in \N$ and $\mu, \rho, d \in (0, 1)$, there exist positive
  constants $\eps(\rho, d, k)$, $\delta(\mu, \rho)$, and $C(\mu, \rho, d, k)$
  such that for every graph $\Gamma$ with $n$ vertices the graph $G = \Gamma
  \cup \Gnp$ w.h.p.\ satisfies the following, provided that $p \geq C/n$.

  Let $1 \leq t \leq \delta n$ and let $\{\mbV_i\}_{i \in [t]}$ be a family of
  $(2k + 4)$-tuples $\mbV_i = (V_i^1, \dotsc, V_i^{2k + 4})$, where $V_i^j
  \subseteq V(G)$ and $|V_i^j| \geq \mu n$ for all $i \in [t]$ and $j \in [2k +
  4]$, and $(V_i^{j_1}, V_i^{j_2})$ is $(\eps, d, \Gamma)$-regular for all
  distinct $j_1, j_2 \in [2k + 4]$ with $|j_1 - j_2| \leq k + 1$. For every
  family of disjoint $(2k + 2)$-cliques $\{\mbs_i, \mbt_i\}_{i \in [t]}$ in $G$
  such that
  \begin{itemize}
    \item $\mbs_i$ is $(\mbV_i^{\leq k + 1}, \rho)$-extendible, and
    \item $\rev(\mbt_i)$ is $(\rev(\mbV_i)^{\leq k + 1}, \rho)$-extendible for
      all $i \in [t]$,
  \end{itemize}
  there exists a collection of disjoint $(2k + 1)$-paths $\{P_i\}_{i \in [t]}$
  in $G$ such that each $P_i$ connects $\mbs_i$ to $\mbt_i$ and is
  $\mbV_i$-bicanonical.
\end{corollary}

\begin{proof}
  Given $k$, $\mu$, $\rho$, and $d$, set $\ell = 2k + 4$ and let
  \[
    \delta = \rho\mu/4, \quad \eps' = \eps_{\ref{lem:path-lemma}}(d/2, k, \ell),
    \quad \eps = \min{\{ d/2, \eps'\rho/2 \}}, \quad \text{and} \quad C =
    C_{\ref{lem:path-lemma}}(\rho\mu/2, d/2, k, \ell).
  \]
  Assume there exists a collection $\{P_i\}_{i \in [z]}$, for some $z < t$, of
  $(2k + 1)$-paths as stated. We show that no matter how such paths are
  constructed, as long as they are $\mbV_i$-bicanonical, we can still find a
  desired $(2k + 1)$-path $P_{z + 1}$. Let
  \[
    N_{s_{z + 1}}^j := N_\Gamma \big( \mbs_{z + 1}^{\geq 2j}, \mbV_{z + 1}^j
    \big) \quad \text{and} \quad N_{t_{z + 1}}^j := N_\Gamma \big( \rev(\mbt_{z
    + 1})^{\geq 2j}, \rev(\mbV_{z + 1})^j \big),
  \]
  for all $j \in [k + 1]$, and note that $|N_{s_{z + 1}}^j| \geq \rho|\mbV_{z +
  1}^j|$ and $|N_{t_{z + 1}}^j| \geq \rho|\rev(\mbV_{z + 1})^j|$. The $(2k +
  1)$-path $P_{z + 1}$ is obtained by applying Lemma~\ref{lem:path-lemma} with
  $\ell$, $\rho\mu/2$ (as $\mu$), $d/2$ (as $d$),
  \[
    \big( N_{s_{z + 1}}^1, \dotsc, N_{s_{z + 1}}^{k + 1}, \mbV_{z + 1}^{k + 2},
    \rev(\mbV_{z + 1})^{k + 2}, N_{t_{z + 1}}^{k + 1}, \dotsc, N_{t_{z + 1}}^1
    \big) \setminus \bigcup_{i \in [z]} V(P_i),
  \]
  (as $V_1, \dotsc, V_\ell$), and $Y_s = Y_t = \varnothing$.

  \begin{assumptions}[Lemma~\ref{lem:path-lemma}]
    Since $t \leq \delta n$ and all $P_i$'s are $\mbV_i$-bicanonical, it follows
    that
    \[
      \Big| N_{s_{z + 1}}^j \setminus \bigcup_{i \in [z]} V(P_i) \Big| \geq \rho
      |V_{z + 1}^j| - 2\delta n \geq (\rho/2)|V_{z + 1}^j| \geq (\rho\mu/2)n,
    \]
    and similarly (with room to spare) for $N_{t_{z + 1}}^j$, $\mbV_{z + 1}^{k +
    2}$, and $\rev(\mbV_{z + 1})^{k + 2}$, due to our choice of constants.
    Therefore, every set used as $V_i$ is a subset of some $V_{z + 1}^j$ of size
    at least $(\rho/2)|V_{z + 1}^j|$. As $(V_{z + 1}^{j_1}, V_{z + 1}^{j_2})$ is
    $(\eps, d, \Gamma)$-regular for distinct $j_1, j_2 \in [2k + 4]$ with $|j_1
    - j_2| \leq k + 1$, we have by Proposition~\ref{prop:large-reg-inheritance}
    that the pair $(V_{j_1}, V_{j_2})$ is $(\eps', d/2, \Gamma)$-regular.
  \end{assumptions}

  This completes the proof.
\end{proof}

\section{The Absorbing-Covering Lemma}

Recall that our proof strategy consists of two steps:
\begin{enumerate*}[label=(\textit{\roman*})]
  \item find an $X$-absorbing path $P$ connecting some $\mbs$ to $\mbt$;
  \item construct a long $(2k + 1)$-path from $\mbt$ to $\mbs$ which contains
    all the vertices of $V(G) \setminus (X \cup V(P))$ and possibly some of $X$.
\end{enumerate*}
In this section we present a lemma which captures the first step and a part of
the second step. Namely, given a set $X$ it provides us with a $(2k + 1)$-path
$P$ which is $X$-absorbing and covers \emph{most} of the vertices. In addition,
the endpoints of $P$ are extendible, which gives us some flexibility as to how
to extend it or combine it with another path in order to close a cycle.

\begin{lemma}[Absorbing-Covering Lemma]\label{lem:absorbing-lemma}
  For all $k, t \in \N$ and $\alpha, \gamma, \mu, d \in (0, 1)$, where $t \geq
  2/\alpha^2$ and $(k + 1) \mid t$, there exist positive constants $\xi(\alpha,
  k)$\footnote{It is crucial for the proof of Theorem~\ref{thm:main-result} that
  $\xi$ does not depend on $\gamma$.}, $\eps(\alpha, \gamma, d, k)$,
  $\rho(\alpha, \gamma, d, k)$, and $C(\alpha, \gamma, \mu, d, k, t)$ such that
  for every graph $\Gamma$ with $n$ vertices the graph $G = \Gamma \cup \Gnp$
  w.h.p.\ satisfies the following, provided that $p \geq C/n$.

  Let $X, W \subseteq V(G)$ be disjoint sets of vertices such that $|W| \geq \mu
  n$, $|X| \leq \xi|W|$, and $\deg_\Gamma(x, W) \geq (\frac{k}{k + 1} +
  \alpha)|W|$ for every $x \in X$. Suppose $(W_i)_{i \in [t]}$ is an
  equipartition of $W$ such that its $(\eps, d, \Gamma)$-reduced graph $R$
  satisfies $\delta(R) \geq (\frac{k}{k + 1} + \alpha)t$ and $(1, \dotsc, t)$ is
  a $k$-cycle in $R$. Then there exists an $X$-absorbing $(2k + 1)$-path $P$
  connecting some $\mbs$ to $\mbt$ with the following properties:
  \begin{enumerate}[label=(\textit{\roman*})]
    \item\label{A-leftover} $|W_i \setminus V(P)| = (1 \pm 0.1)\gamma|W_i|$, for
      every $i \in [\ell]$,
    \item\label{A-cliques-s} $\rev(\mbs)$ is $(\mbW_s, \rho)$-extendible, where
      $\mbW_s = (W_z, W_{k + 1}, \dotsc, W_2) \setminus V(P)$, and
    \item\label{A-cliques-t} $\mbt$ is $(\mbW_t, \rho)$-extendible, where
      $\mbW_t = (W_z, W_1, \dotsc, W_k) \setminus V(P)$,
  \end{enumerate}
  for some $z \in \bigcap_{i \in [k + 1]} N_R(i)$.
\end{lemma}

Before we prove the lemma, we present a `local version' of the lemma which
serves as the central technical piece towards that goal.

\begin{lemma}[Local Absorbing-Covering Lemma]\label{lem:local-absorbing-lemma}
  For all $k \in \N$ and $\alpha, \gamma, \lambda, \mu, d \in (0, 1/2)$, there
  exist positive constants $\delta(\alpha)$, $\xi(\alpha)$, $\eps(\alpha,
  \gamma, \lambda, d, k)$, $\rho(\alpha, \gamma, \lambda, d, k)$, and $C(\alpha,
  \gamma, \lambda, \mu, d, k)$ such that for every graph $\Gamma$ with $n$
  vertices the graph $G = \Gamma \cup \Gnp$ w.h.p.\ satisfies the following,
  provided that $p \geq C/n$.

  Let $V_1, \dotsc, V_{k + 1}, X, Y \subseteq V(G)$ be disjoint sets of vertices
  satisfying
  \begin{itemize}
    \item $|V_1| = \dotsb = |V_{k + 1}| = |Y| \geq \mu n$ and $|X| \leq
      \xi|V_i|$,
    \item for every $x \in X$ and $i \in [k + 1]$, we have $\deg_\Gamma(x, V_i)
      \geq \alpha|V_i|$, and
    \item $(V_i, V_j)$, $(V_i, Y)$ are $(\eps, d, \Gamma)$-regular for all
      distinct $i, j \in [k + 1]$.
  \end{itemize}

  Then for every $Q \subseteq V(G)$ such that $|Y \setminus Q| \geq \lambda|Y|$
  and $|V_i \cap Q| \leq \delta|V_i|$ for all $i \in [k + 1]$, there exists an
  $X$-absorbing $(2k + 1)$-path $P$ with $V(P) \subseteq (V_1 \cup \dotsb \cup
  V_{k + 1} \cup Y) \setminus Q$ connecting some $\mbs$ to $\mbt$ with the
  following properties:
  \begin{enumerate}[label=(\textit{\roman*})]
    \item\label{C-leftover} $|V_i \setminus (V(P) \cup Q)| = (1 \pm
      0.001)\gamma|V_i|$, for every $i \in [k + 1]$,
    \item\label{C-small-usage} $|Y \cap V(P)| \leq \lambda |Y|$,
    \item\label{C-cliques-s} $\rev(\mbs)$ is $(\mbS, \rho)$-extendible, where
      $\mbS = (Y, V_{k + 1}, \dotsc, V_2) \setminus (V(P) \cup Q)$, and
    \item\label{C-cliques-t} $\mbt$ is $(\mbT, \rho)$-extendible, where $\mbT =
      (Y, V_1, \dotsc, V_k) \setminus (V(P) \cup Q)$.
  \end{enumerate}
\end{lemma}

\begin{proof}
  Given $k$, $\alpha$, $\gamma$, $\lambda$, $\mu$, and $d$, let $L = \ceil*{8(1
  - \gamma)/\lambda^2}$, $\ell_1 = 2L(k + 1)$, and $\ell_2 = (2L + 1)(k + 1) +
  1$. Furthermore, we use
  \begin{gather*}
    \eta = \min{\{ \alpha, \gamma, \lambda \}}, \quad \xi = \alpha/48, \quad
    \delta = \min{\{1/10^4, \alpha/8 \}}, \\
    \rho' = \min{\{ \eta/10^4, \rho_{\ref{lem:path-lemma}}(d/2, k, \ell_1),
    \rho_{\ref{lem:path-lemma}}(d/2, k, \ell_2) \}}, \quad \rho = \rho'\eta^2/8,
    \quad \mu' = \rho'\eta^3 \mu/8, \\
    \eps' = \min{\{ \eps_{\ref{lem:path-lemma}}(d/2, k, \ell_1),
    \eps_{\ref{lem:path-lemma}}(d/2, k, \ell_2) \}}, \quad \eps = \min{\{ d/2,
    \eps'\rho'\eta^2/8 \}},
  \end{gather*}
  and $C = \max{\{ C_{\ref{lem:path-lemma}}(\mu', d/2, k, \ell_1),
  C_{\ref{lem:path-lemma}}(\mu', d/2, k, \ell_2) \}}$. Throughout we always
  apply Lemma~\ref{lem:path-lemma} with $k$, $\mu'$ (as $\mu$) and $d/2$ (as
  $d$), and thus omit explicitly stating the parameters.

  Let $V_i' \subseteq V_i \setminus Q$ and $Y' \subseteq Y \setminus Q$ be
  arbitrary sets of size $(1 - \delta)|V_i|$ and $\lambda|Y|$, and set $m =
  \floor*{\big( (1 - \gamma)|V_i'| + 2 \big) / (4L + 2)}$. In order to show the
  desired statement we successively apply Lemma~\ref{lem:path-lemma} to the
  appropriate subsets of $V_i'$ and $Y'$ to construct paths $P_1 \subseteq
  \dotsb \subseteq P_m$ where each $P_j$, $j \in [m]$, intersects every $V_i'$
  in exactly $4L \cdot j + 2(j - 1)$ vertices, $Y'$ in at most $2(j - 1)$
  vertices, and contains no other vertices of $V(G)$. This reveals how the
  constant $L$ helps us towards satisfying \ref{C-small-usage}. In particular,
  such a path $P_m$ satisfies \ref{C-leftover} and \ref{C-small-usage}.

  Let $X_1 \cup \dotsb \cup X_r = X$ be an arbitrary partition of $X$ into sets
  of size exactly $L$ (the last set may be smaller, however this does not affect
  any part of the argument). Consider the vertices $x_1, \dotsc, x_L \in X_1$.
  Let $N_i(x) := N_\Gamma(x, V_i')$, for $i \in [k + 1]$ and $x \in X_1$, and
  let $\mbN(x) = (N_1(x), \dotsc, N_{k + 1}(x))$. We apply
  Lemma~\ref{lem:path-lemma} with $\ell_1$ (as $\ell$),
  \[
    \big( \mbN(x_1), \mbN(x_1), \mbN(x_2), \mbN(x_2), \dotsc, \mbN(x_L),
    \mbN(x_L) \big)
  \]
  (as $V_1, \dotsc, V_\ell$), and $Y'$ (as both $Y_s$ and $Y_t$) to obtain a
  $(2k + 1)$-path $P_1$ connecting some $\mbs$ to $\mbt_1$. By construction
  $P_1$ is $X_1$-absorbing (see Figure~\ref{fig:absorbing-path}).

  \begin{assumptions}[Lemma~\ref{lem:path-lemma}]
    Since $|N_i(x)| \geq \alpha|V_i| - \delta|V_i| \geq (\alpha/2)|V_i| \geq
    (\alpha/2)\mu n$ and $|Y'| \geq \lambda|Y| \geq \lambda\mu n$, by the
    assumptions of the lemma, and $(V_i, V_j)$ and $(V_i, Y)$ are $(\eps, d,
    \Gamma)$-regular for all distinct $i, j \in [k + 1]$,
    Proposition~\ref{prop:large-reg-inheritance} shows that their respective
    subsets of size at least $(\alpha/2)|V_i|$ and $\lambda|Y|$ form $(\eps',
    d/2, \Gamma)$-regular pairs.
  \end{assumptions}

  \begin{figure}[!htbp]
    \centering
    \includegraphics[scale=0.8]{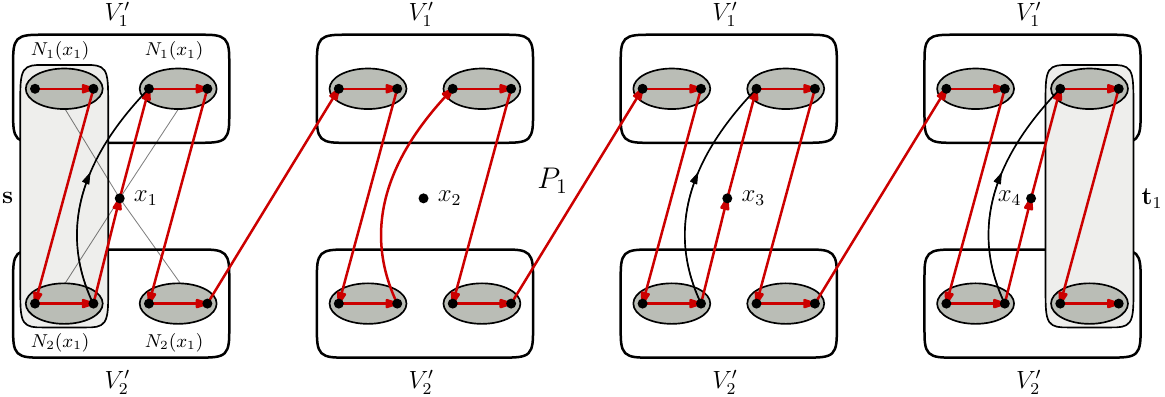}
    \caption{An example of an $X_1$-absorbing $(2k + 1)$-path $P_1$ connecting
    $\mbs$ to $\mbt_1$ for $k = 1$. The $(2k + 1)$-path including $V(P_1) \cup
    \{x_1, x_3, x_4\}$ is given in red. For simplicity, we depict only skeletons
    of the mentioned $(2k + 1)$-paths.}
    \label{fig:absorbing-path}
  \end{figure}

  As $\eta \leq \min{\{\alpha, \lambda\}}$, by
  Lemma~\ref{lem:path-lemma}\ref{s-extendible}, $\rev(\mbs)$ is $(\mbV_s
  \setminus V(P_1), \rho'\eta/2)$-extendible where $\mbV_s = (Y', V_{k + 1}',
  \dotsc, V_2')$. Let $S_i \subseteq N_\Gamma(\rev(\mbs)^{\geq 2i}, \mbV_s^i)
  \setminus V(P_1)$, $i \in [k + 1]$, be arbitrary sets of size $|S_i| =
  (\rho'\eta/4)|\mbV_s^i \setminus V(P_1)|$, and set $S := \bigcup_{i \in [k +
  1]} S_i$. Moreover, we conclude that $\mbt_1$ is $(\mbV_t \setminus (V(P_1)
  \cup S), \rho'\eta/4)$-extendible, where $\mbV_t = (Y', V_1', \dotsc, V_k')$,
  due to Lemma~\ref{lem:path-lemma}\ref{t-extendible} and as
  \[
    \frac{\rho'\eta}{2}|V_i' \setminus V(P_1)| - |V_i' \cap S| \geq
    \frac{\rho'\eta}{2}|V_i' \setminus V(P_1)| - \frac{\rho'\eta}{4}|V_i'
    \setminus V(P_1)| \geq \frac{\rho'\eta}{4}|V_i' \setminus V(P_1)|
  \]
  for all $i \in [k + 1]$, and similarly for $Y' \setminus (V(P_1) \cup S)$.

  The goal is to gradually extend the path $P_1$ into paths $P_1 \subseteq P_2
  \subseteq \dotsb \subseteq P_m$ such that $P_j$ is $(\bigcup_{1 \leq j' \leq
  j} X_{j'})$-absorbing, for $j \in [r]$, and the paths $P_{r + 1} \subseteq
  \dotsb \subseteq P_m$ progressively cover the remaining vertices of $V_i$'s.
  More precisely, we show by induction on $j \in [m]$ that there exists a path
  $P_j \subseteq (V_1' \cup \dotsb \cup V_{k + 1}' \cup Y') \setminus S$ such
  that
  \begin{enumerate}[(C\arabic*), font=\itshape, leftmargin=2.8em]
    \item\label{C-inv-path} $P_j$ is $(\bigcup_{1 \leq i \leq \min{\{j, r\}}}
      X_i)$-absorbing and connects $\mbs$ to some $\mbt_j$,
    \item\label{C-inv-cover} $|V(P_j) \cap V_i| = 4L \cdot j + 2(j - 1)$, for
      all $i \in [k + 1]$, and $|V(P_j) \cap Y| \leq 2(j - 1)$, and
    \item\label{C-inv-extendible} $\mbt_j$ is $(\mbV_t \setminus (V(P_j) \cup
      S), \rho'\eta/4)$-extendible.
  \end{enumerate}
  Hence, for $j = m$ we obtain an $X$-absorbing $(2k + 1)$-path $P_m$ connecting
  $\mbs$ to $\mbt_m$ and which covers exactly $\floor{(1 - \gamma)|V_i'|}$
  vertices in every $V_i$. Consequently, $|V_i \setminus (V(P_m) \cup Q)| = (1
  \pm 0.001)\gamma|V_i|$. As for \ref{C-small-usage}, \ref{C-inv-cover}
  implies that at most $2m \leq (1 - \gamma)|Y|/L \leq \lambda |Y|$ vertices of
  $Y$ are used. In order to see \ref{C-cliques-s} recall that the sets $S_i$
  are of size $(\rho'\eta/4)|V_i' \setminus V(P_1)| \geq (\rho'\eta/4)(1 -
  \delta)|V_i| - 4L \geq (\rho'\eta^2/8)|V_i|$ and $(\rho'\eta/4)|Y'| - 4L \geq
  (\rho'\eta/8)\lambda|Y| \geq (\rho'\eta^2/8)|Y|$, which in turn translates
  into being $(\mbS, \rho)$-extendible due to our choice of constants (in
  particular, $\rho = \rho'\eta^2/8$). Lastly, \ref{C-cliques-t} follows from
  \ref{C-inv-extendible} and our choice of constants again.

  Note that for $j = 1$ the path $P_1$ satisfies
  \ref{C-inv-path}--\ref{C-inv-extendible}. Assume that the induction hypothesis
  holds for some $j - 1$, $j > 1$, and let us show it for $j$. Let $T_i =
  N_\Gamma(\mbt_{j - 1}^{\geq 2i}, \mbV_t^i)$, for all $i \in [k + 1]$. The
  argument slightly differs depending on whether $j \leq r$ or $r < j$.

  \textbf{Case $j \leq r$.} Consider the vertices $x_1, \dotsc, x_L \in X_j$. We
  apply Lemma~\ref{lem:path-lemma} with $\ell_2$ (as $\ell$),
  \[
    \big( \underset{T_1 \subseteq Y'}{T_1}, \dotsc, \underset{T_{k + 1}
    \subseteq V_k'}{T_{k + 1}}, V_{k + 1}', \mbN(x_1), \mbN(x_1), \mbN(x_2),
    \mbN(x_2), \dotsc, \mbN(x_L), \mbN(x_L) \big) \setminus (V(P_{j - 1}) \cup
    S)
  \]
  (as $V_1, \dotsc, V_\ell$), and $Y' \setminus V(P_{j - 1})$ (as both $Y_s$ and
  $Y_t$) to obtain a $(2k + 1)$-path $P'$ connecting $\mbt_{j - 1}$ to some
  $\mbt_j$. By construction $P'$ is $X_j$-absorbing. One easily checks that
  \ref{C-inv-path}--\ref{C-inv-extendible} hold for $P_j := P_{j - 1} \cup P'$.

  \begin{assumptions}[Lemma~\ref{lem:path-lemma}]
    Observe that by assumption of the lemma $\deg_\Gamma(x, V_i) \geq
    \alpha|V_i|$, and the fact that $|S_i| = (\rho'\eta/4)|\mbV_s^i \setminus
    V(P_1)|$, we have
    \begin{equation}\label{eq:local-cov-sets-lb-neighbourhoods}
      \big| N_i(x) \setminus (V(P_{j - 1}) \cup S) \big| \geq \alpha|V_i| -
      \delta|V_i| - r \cdot 6L - (\rho'\eta/4)|V_i| \geq (\alpha - \delta - 6\xi
      - \rho'\eta/4)|V_i| \geq (\eta/2)|V_i|,
    \end{equation}
    by our choice of constants. Similarly,
    \begin{equation}\label{eq:local-cov-sets-lb-Y}
      \big| Y' \setminus (V(P_{j - 1}) \cup S) \big| \geq \lambda|Y| - 2r -
      (\rho'\eta/4)|Y| \geq (\lambda - 2\xi/L - \rho'\eta/4)|Y| \geq
      (\eta/2)|Y|.
    \end{equation}
    Lastly,
    \begin{equation}\label{eq:local-cov-sets-lb-Ti}
      \begin{aligned}
        \big| T_i \setminus (V(P_{j - 1}) \cup S) \big| &\geq (\rho'\eta/4)
        |V_i' \setminus (V(P_{j - 1}) \cup S)| \geq (\rho'\eta/4) \big(1 - (1 -
        \gamma) - \rho'\eta/4 \big)|V_i'| \\
        & \geq (\rho'\eta/4)(\gamma - \rho'\eta/4)|V_i'| \geq
        (\rho'\eta^2/8)|V_i'|,
      \end{aligned}
    \end{equation}
    for all $2 \leq i \leq k + 1$, and
    \begin{equation}\label{eq:local-cov-sets-lb-T1}
      \begin{aligned}
        \big| T_1 \setminus (V(P_{j - 1}) \cup S) \big| &\geq (\rho'\eta/4)|Y'
        \setminus (V(P_{j - 1}) \cup S)| \geq (\rho'\eta/4)\big(1 - (1 -
        \lambda) - \rho'\eta/4 \big)|Y'| \\
        & \geq (\rho'\eta/4)(\lambda - \rho'\eta/4)|Y'| \geq
        (\rho'\eta^2/8)|Y'|.
      \end{aligned}
    \end{equation}
    As $(V_i, V_j)$ and $(V_i, Y)$ are $(\eps, d, \Gamma)$-regular for all
    distinct $i, j \in [k + 1]$, by Proposition~\ref{prop:large-reg-inheritance}
    their respective subsets of sizes as in
    \eqref{eq:local-cov-sets-lb-neighbourhoods}--\eqref{eq:local-cov-sets-lb-T1}
    form $(\eps', d/2, \Gamma)$-regular pairs.
  \end{assumptions}

  \textbf{Case $r < j$.} Let $\mbV = (V_1', \dotsc, V_{k + 1}')$. In this case
  we apply Lemma~\ref{lem:path-lemma} with $\ell_2$ (as $\ell$),
  \[
    (\underset{T_1 \subseteq Y'}{T_1}, \underset{T_2 \subseteq V_1'}{T_2},
    \dotsc, \underset{T_{k + 1} \subseteq V_k'}{T_{k + 1}}, V_{k + 1}',
    \underbrace{\mbV, \mbV, \dotsc, \mbV}_{\text{$2L$ times}}) \setminus (V(P_{j
    - 1}) \cup S)
  \]
  (as $V_1, \dotsc, V_\ell$), $Y_s = \varnothing$, and $Y'$ (as $Y_t$), in order
  to obtain a $(2k + 1)$-path $P'$ connecting $\mbt_{j - 1}$ to some $\mbt_j$.
  Setting $P_j := P_{j - 1} \cup P'$ we get a $(2k + 1)$-path which satisfies
  \ref{C-inv-path}--\ref{C-inv-extendible}. Checking that the assumptions of
  Lemma~\ref{lem:path-lemma} are satisfied is analogous to the previous
  application (see
  \eqref{eq:local-cov-sets-lb-neighbourhoods}--\eqref{eq:local-cov-sets-lb-T1}),
  and is thus omitted.
\end{proof}

With this at hand we give the proof of the Absorbing-Covering Lemma
(Lemma~\ref{lem:absorbing-lemma}).

\begin{proof}[Proof of Lemma~\ref{lem:absorbing-lemma}]
  We use numerous constants which we pin down next. For easier reference, set
  $\ell = k + 1$ and $r = t/(k + 1)$. Take $\gamma'$ such that it simultaneously
  satisfies $\gamma' \leq \gamma$ and $0.98\gamma' > 0.9 \gamma$, and let
  \begin{gather*}
    \xi' = \xi_{\ref{lem:local-absorbing-lemma}}(\alpha), \quad \delta =
    \delta_{\ref{lem:local-absorbing-lemma}}(\alpha), \quad \xi = \alpha
    \xi'/(k + 1), \quad \lambda = \min{\{
    \alpha\delta/2, \alpha\gamma'/10^3 \}}, \quad \mu' = \gamma'\mu/(2t), \\
    \rho' = \min{\{ \alpha\delta/4, \alpha\gamma'/16,
    \rho_{\ref{lem:path-lemma}}(d/2, k, \ell),
    \rho_{\ref{lem:local-absorbing-lemma}}(\alpha, \gamma', \lambda, d/2, k) \}},
    \quad \rho = \rho'\gamma'/4, \\
    \eps' = \min{\{ \eps_{\ref{lem:path-lemma}}(d/2, k, \ell),
    \eps_{\ref{cor:connecting-lemma}}(\rho'\gamma'/2, d/2, k),
    \eps_{\ref{lem:local-absorbing-lemma}}(\alpha, \gamma', \lambda, d/2, k) \}},
    \quad \eps = \min{\{ d/2, \eps'\gamma'/2 \}}, \quad \text{and} \\
    C = \max{\{ C_{\ref{lem:path-lemma}}(\mu', d/2, k, \ell),
    C_{\ref{cor:connecting-lemma}}(\mu', \rho'\gamma'/2, d/2, k),
    C_{\ref{lem:local-absorbing-lemma}}(\alpha, \gamma', \lambda, \mu', d/2, k)
    \}}.
  \end{gather*}
  We always apply Lemma~\ref{lem:path-lemma} with $\ell$, $k$, $\mu'$ (as
  $\mu$), and $d/2$ (as $d$), Corollary~\ref{cor:connecting-lemma} with $k$,
  $\mu'$ (as $\mu$), $\rho'\gamma'/2$ (as $\rho$), and $d/2$ (as $d$), and
  Lemma~\ref{lem:local-absorbing-lemma}, with $k$, $\alpha$, $\gamma'$ (as
  $\gamma$), $\lambda$, $\mu'$ (as $\mu$), and $d/2$ (as $d$), and thus omit
  explicitly stating the parameters.

  Let $\mbK_j = (W_{(j - 1)(k + 1) + 1}, \dotsc, W_{(j - 1)(k + 1) + k + 1})$,
  for all $j \in [r]$, and note that every such tuple is $(\eps, d,
  \Gamma)$-regular due to $(1, \dotsc, t)$ being a $k$-cycle in $R$.
  Furthermore, let $\phi \colon X \to [r]$ and $\psi \colon [r] \to [t]$ be
  functions such that:
  \begin{enumerate}[(F\arabic*), font=\itshape, leftmargin=2.8em]
    \item\label{phi} for every $x \in X$ we have $\deg_\Gamma(x,
      \mbK_{\phi(x)}^i) \geq \alpha|W|/t$, for all $i \in [k + 1]$,
    \item\label{phi-small} $|\phi^{-1}(j)| \leq |X|/(\alpha r)$, for all $j \in
      [r]$,
    \item\label{psi} $(\mbK_j, W_{\psi(j)})$ is an $(\eps, d, \Gamma)$-regular
      $(k + 2)$-tuple, for all $j \in [r]$, and
    \item\label{psi-small} $|\psi^{-1}(i)| \leq 1/\alpha$, for all $i \in [t]$.
  \end{enumerate}
  We show that such functions exist. Consider some $x \in X$ and let $\hbar \in
  [0, 1]$ be a fraction of tuples $\mbK_j$ for which \ref{phi} holds. Since
  $\deg_\Gamma(x, W) \geq (\frac{k}{k + 1} + \alpha)|W|$ and each $W_i$ is of
  size $|W|/t$, we derive
  \[
    \hbar r \cdot (k + 1)\frac{|W|}{t} + (1 - \hbar)r \cdot (k +
    \alpha)\frac{|W|}{t} \geq \Big( \frac{k}{k + 1} + \alpha \Big)|W|.
  \]
  Direct computation gives $\hbar \geq (k + 1)\alpha - \alpha \geq \alpha$. In
  conclusion, for every $x \in X$ there are at least $\alpha r$ values $j \in
  [r]$ such that setting $\phi(x) := j$ satisfies \ref{phi}. Consequently, a
  simple averaging argument shows \ref{phi-small}. A similar calculation shows
  that a function $\psi$ exists and the proof is thus omitted. Lastly, let $X_1
  \cup \dotsb \cup X_r = X$ be a partition given by $X_j = \phi^{-1}(j)$.

  We show by induction on $q \in [r]$ that there exist vertex-disjoint $(2k +
  1)$-paths $P_1, \dotsc, P_q$ such that for every $j \in [q]$ the path $P_j
  \subseteq \mbK_j \cup W_{\psi(j)}$ connects some $\mbs_j$ to $\mbt_j$ and
  satisfies the following properties:
  \begin{enumerate}[label=(I\arabic*), font=\itshape, leftmargin=2.8em]
    \item\label{I-absorbing} the $(2k + 1)$-path $P_j$ is $X_j$-absorbing,
    \item\label{I-leftover} $|\mbK_j^i \setminus V(P_j)| = (1 \pm
      0.001)\gamma'|W|/t$, for all $i \in [k + 1]$,
    \item\label{I-leftover-others} $|W_{\psi(j)} \cap V(P_j)| \leq
      \lambda|W|/t$, and
    \item\label{I-extendible} $\rev(\mbs_j)$ is $(\mbS_j, \rho')$-extendible and
      $\mbt_j$ is $(\mbT_j, \rho')$-extendible, where
      \[
        \mbS_j = \big( W_{\psi(j)}, \rev(\mbK_j^{\geq 2}) \big) \setminus
        \bigcup_{j' \in [q]} V(P_{j'}) \quad \text{and} \quad \mbT_j =
        (W_{\psi(j)}, \mbK_j^{\leq k}) \setminus \bigcup_{j' \in [q]} V(P_{j'}).
      \]
  \end{enumerate}

  Consider the base of the induction $q = 1$. We apply
  Lemma~\ref{lem:local-absorbing-lemma} with $W_1, \dotsc, W_{k + 1}$ (as $V_1,
  \dotsc, V_{k + 1}$), $X_1$ (as $X$), $W_{\psi(1)}$ (as $Y$), and $Q =
  \varnothing$, in order to obtain an $X_1$-absorbing $(2k + 1)$-path $P_1$
  connecting some $\mbs_1$ to $\mbt_1$, with $\rev(\mbs_1)$ being $(\mbS_1,
  \rho')$-extendible and $\mbt_1$ being $(\mbT_1, \rho')$-extendible, and thus
  trivially satisfying all \ref{I-absorbing}--\ref{I-extendible}.

  Assume that the induction hypothesis holds for some $q - 1$, $q > 1$, and let
  us show it for $q$. First, for all $j \in [q - 1]$ and $i \in [k + 1]$, let
  \[
    N_{s_j}^i \subseteq N_\Gamma(\rev(\mbs_j)^{\geq 2i}, \mbS_j^i) \quad
    \text{and} \quad N_{t_j}^i \subseteq N_\Gamma(\mbt_j^{\geq 2i}, \mbT_j^i)
  \]
  be sets of size $|N_{s_j}^i| = \rho'|\mbS_j^i|$ and $|N_{t_j}^i| =
  \rho'|\mbT_j^i|$. This helps us maintain the property \ref{I-extendible}, for
  all $j \in [q - 1]$. Next, we use Lemma~\ref{lem:local-absorbing-lemma} with
  $\mbK_q$ (as $V_1, \dotsc, V_{k + 1}$), $X_q$ (as $X$), $W_{\psi(q)}$ (as
  $Y$), and
  \begin{equation}\label{eq-Q-size}
    Q := \bigcup_{j \in [q - 1], i \in [k + 1]} V(P_j) \cup N_{s_j}^i \cup
    N_{t_j}^i.
  \end{equation}
  We obtain a $(2k + 1)$-path $P_q$ connecting $\mbs_q$ to $\mbt_q$, with
  $\rev(s_q)$ being $(\mbS_q, \rho')$-extendible and $\mbt_q$ is $(\mbT_q,
  \rho')$-extendible, establishing \ref{I-absorbing}--\ref{I-extendible}.
  Crucially, \ref{I-extendible} remains to hold for $j < q$ as $P_q$ does not
  intersect the set $Q$.

  \begin{assumptions}[Lemma~\ref{lem:local-absorbing-lemma}]
    Firstly, for all $j \in [r]$ we have
    \[
      |X_q| \leq |X|/(\alpha r) \leq (k + 1)\xi|W|/(\alpha t) \leq \xi'|W|/t.
    \]
    Secondly, for every $x \in X_q$ we have $\deg_\Gamma(x, \mbK_q^i) \geq
    \alpha|W|/t$, for all $i \in [k + 1]$, by \ref{phi}. Thirdly, we know that
    $(\mbK_q, W_{\psi(q)})$ is an $(\eps, d, \Gamma)$-regular $(k + 2)$-tuple by
    \ref{psi}. Lastly, we know that for all $i \in [k + 1]$,
    \begin{equation}\label{eq:Q-lb-Wi}
      |W_{(q - 1)(k + 1) + i} \setminus Q| \geq |W|/t - (1/\alpha) \cdot
      (\lambda|W|/t + 2\rho'|W|/t) \geq (1 - \delta)|W|/t,
    \end{equation}
    by the choice of constants. Similarly,
    \begin{equation}\label{eq:Q-lb-psi}
      |W_{\psi(q)} \setminus Q| \geq 0.9\gamma'|W|/t - (1/\alpha) \cdot
      (\lambda|W|/t + 2\rho'|W|/t) \geq (\gamma'/2)|W|/t \geq \lambda|W|/t,
    \end{equation}
    again by our choice of constants.
  \end{assumptions}

  Before we patch all the paths together into one long $X$-absorbing path which
  covers almost everything, we need to make sure that its endpoints are $(\mbS,
  \rho)$-extendible and $(\mbT, \rho)$-extendible. As before, let
  \[
    N_{s_j}^i \subseteq N_\Gamma(\rev(\mbs_j)^{\geq 2i}, \mbS_j^i) \quad
    \text{and} \quad N_{t_j}^i \subseteq N_\Gamma(\mbt_j^{\geq 2i}, \mbT_j^i)
  \]
  for all $j \in [r]$ and $i \in [k + 1]$, be sets of size $|N_{s_j}^i| =
  \rho'|\mbS_j^i|$ and $|N_{t_j}^i| = \rho'|\mbT_j^i|$, and
  \[
    Q' := \bigcup_{j \in [r], i \in [k + 1]} V(P_j) \cup N_{s_j}^i \cup
    N_{t_j}^i.
  \]
  We apply Lemma~\ref{lem:path-lemma} with $(W_1, \dotsc, W_{k + 1}) \setminus
  Q'$ (as $V_1, \dotsc, V_\ell$), and $W_{\psi(1)} \setminus Q'$ (as both $Y_s$
  and $Y_t$). This yields a $(2k + 2)$-clique $\mbt$ which is $(\mbT \setminus
  Q', \rho'\gamma'/2)$-extendible.

  \begin{assumptions}[Lemma~\ref{lem:path-lemma}]
    One easily checks that $|W_i \setminus Q'| \geq (\gamma'/2)|W|/t$ and
    $|W_{\psi(1)} \setminus Q'| \geq (\gamma'/2)|W|/t$, similarly as in
    \eqref{eq:Q-lb-Wi} and \eqref{eq:Q-lb-psi}. Therefore, it follows from
    \ref{psi} and Proposition~\ref{prop:large-reg-inheritance}, that they form
    an $(\eps', d/2, \Gamma)$-regular $(k + 2)$-tuple.
  \end{assumptions}

  Finally, we apply Corollary~\ref{cor:connecting-lemma} with
  \[
    \big\{ (\mbT_j, \mbK_j^{k + 1} \setminus Q', \mbK_{j + 1}^1 \setminus Q',
    \mbS_{j + 1}) \big\}_{j \in [r - 1]}\ \cup \big\{ (\mbT_r, \mbK_r^{k + 1}
    \setminus Q', \mbK_1^1 \setminus Q', \mbT) \big\}
  \]
  (as $\{\mbV_i\}$), and $\{\mbt_j, \mbs_{j + 1}\}_{j \in [r - 1]} \cup
  \{\mbt_r, \mbt\}$ (as $\{\mbs_i, \mbt_i\}$), to obtain a $(2k + 1)$-path $P$
  connecting $\mbs := \mbs_1$ to $\mbt$.

  It remains to show that $P$ satisfies
  Lemma~\ref{lem:absorbing-lemma}\ref{A-leftover}--\ref{A-cliques-t}.  Property
  \ref{I-leftover} directly settles the upper bound in
  Lemma~\ref{lem:absorbing-lemma}\ref{A-leftover}. As for the lower bound,
  observe that we may have used some additional vertices of $W_i$ as $Y$ in
  Lemma~\ref{lem:local-absorbing-lemma}, in case $\psi(j) = i$. However, this
  happens at most $1/\alpha$ times, by \ref{phi-small}, and every time at most
  $\lambda |W|/t$ vertices of it are used, due to
  Lemma~\ref{lem:local-absorbing-lemma}\ref{C-small-usage}. Lastly, at most $(r
  + 1) \cdot 2(2k + 4)$ vertices are used by the applications of
  Lemma~\ref{lem:path-lemma} and Corollary~\ref{cor:connecting-lemma}.
  Therefore,
  \[
    |W_i \setminus V(P)| \geq (1 - 0.001)\gamma'|W|/t - 1/\alpha \cdot \lambda
    |W|/t - (r + 1) \cdot 2(2k + 4) \geq (1 - 0.1)\gamma|W|/t,
  \]
  for all $i \in [t]$, by our choice of constants. The properties
  \ref{A-cliques-s} and \ref{A-cliques-t}, follow directly from
  \ref{I-extendible}, the observation that additionally at most $(r + 1) \cdot
  2(2k + 4)$ vertices are used in order to join the paths $P_j$ into $P$, and
  our choice of constants. Setting $z = \psi(1)$ completes the proof.
\end{proof}

\section{Proof of Theorem \ref{thm:main-result}}

Throughout the proof we make use of several constants which we pin down next.
Given $\alpha > 0$ and $k \in \N$, we let
\begin{gather*}
  d = \alpha/4, \qquad \xi = \xi_{\ref{lem:absorbing-lemma}}(\alpha/4, k),
  \qquad \gamma = \xi^2/8, \qquad
  \rho = \rho_{\ref{lem:absorbing-lemma}}(\alpha/4, \gamma, d/2, k), \\
  \eps'' = \min{\{ d/4, \eps_{\ref{cor:connecting-lemma}}(\rho, d/4, k) \}},
  \qquad \eps' = \min{\{ \eps''\gamma/2,
  \eps_{\ref{lem:absorbing-lemma}}(\alpha/4, \gamma, d/2, k) \}}, \\
  \eps = \eps'\xi^2/16, \qquad m = \max{\{n_{0_{\ref{thm:KSS}}}(k),
  32/\alpha^2\}}, \qquad \text{and} \qquad t =
  t_{\ref{thm:degree-form-regularity}}(\eps, d, m).
\end{gather*}
Lastly, set
\[
  \mu = \frac{(1 - \eps)\xi^3}{16t} \qquad \text{and} \qquad C = \max{\{
  C_{\ref{cor:connecting-lemma}}(\mu, \rho, d/4, k),
  C_{\ref{lem:absorbing-lemma}}(\alpha/4, \gamma, \mu, d/2, k) \}}.
\]
We make use of Corollary~\ref{cor:connecting-lemma} with $k$, $\mu$, $\rho$, and
$d/4$ (as $d$), and Lemma~\ref{lem:absorbing-lemma} with $k$, $t$, $\alpha/4$
(as $\alpha$), $\gamma$, $\mu$, and $d/2$ (as $d$), thus omit explicitly stating
the parameters.

Suppose $n$ is sufficiently large and let $\Gamma$ be a graph with $n$ vertices
and $\delta(\Gamma) \geq (\frac{k}{k + 1} + \alpha)n$. Set $G := \Gamma \cup
\Gnp$, for some $p \geq C/n$, and observe that $G$ has the properties as given
by Corollary~\ref{cor:connecting-lemma} and Lemma~\ref{lem:absorbing-lemma} with
high probability, for the parameters specified above. We show that such a graph
$G$ contains the $(2k + 1)$-st power of a Hamilton cycle.

Let $\cV = (V_i)_{i = 0}^{t}$ be an $\eps$-regular partition obtained by
applying the regularity lemma (Theorem~\ref{thm:degree-form-regularity}) to
$\Gamma$, let $\Gamma'$ be the spanning subgraph satisfying
Theorem~\ref{thm:degree-form-regularity}\ref{reg-except}--\ref{reg-density},
and let $R$ be the $(\eps, d, \Gamma', \cV)$-reduced graph. Without loss of
generality we may assume that $(k + 1) \mid t$. A well-known property of the
reduced graph is that it `inherits' the minimum degree of the graph $\Gamma'$,
namely
\[
  \delta(R) \geq \Big( \frac{k}{k + 1} + \frac{\alpha}{4} \Big) t.
\]
In order to see this, observe that if there is a partition class $V_i$ which
does not satisfy the above, then by
Theorem~\ref{thm:degree-form-regularity}\ref{reg-empty}--\ref{reg-density} for
every $v \in V_i$ we have
\[
  \deg_{\Gamma'}(v) < \Big( \frac{k}{k + 1} + \frac{\alpha}{4} \Big)t \cdot
  \frac{n}{t} + \eps n \leq \Big( \frac{k}{k + 1} + \frac{\alpha}{2} \Big)n \leq
  \deg_{\Gamma}(v) - (d + \eps)n,
\]
contradicting Theorem~\ref{thm:degree-form-regularity}\ref{reg-degree}, due to
our choice of constants. Therefore, as $t \geq m = n_{0_{\ref{thm:KSS}}}(k)$, by
Theorem~\ref{thm:KSS}, $R$ contains a spanning $k$-cycle. By relabelling the
vertex set of $R$ we may assume that $(1, \dotsc, t)$ is such a $k$-cycle.

Write $\tilde n := |V_i|$ and recall that $\tilde n \in [(1 - \eps)n/t, n/t]$.
For all $i \in [t]$, let $X_i \subseteq V_i$ be a subset of size $(\xi/2) \tilde
n$ chosen uniformly at random and set $W_i := V_i \setminus X_i$. From
Proposition~\ref{prop:large-reg-inheritance} it follows that, crucially, the new
partitions $\cW = (W_i)_{i \in [t]}$ and $\cX = (X_i)_{i \in [t]}$ are such that
$R$ is a subgraph of both the $(\eps', d/2, \Gamma', \cW)$-reduced graph as well
as the $(\eps', d/2, \Gamma', \cX)$-reduced graph as $\max{\{\eps/(1 - \xi/2),
2\eps/\xi\}} < \eps'$. In particular, if $(V_i, V_j)$ is $(\eps, d,
\Gamma')$-regular then $(U_i, U_j)$ is $(\eps', d/2, \Gamma')$-regular, where $U_i
\in \{X_i, W_i\}$ and $U_j \in \{X_j, W_j\}$. Let $X = \bigcup_{i \in [t]} X_i$
and $W = \bigcup_{i \in [t]} W_i$, and note that $|W| \geq (1 - \eps - \xi/2)n
\geq \mu n$ (with room to spare) and $|X| \leq \xi|W|$. As $\delta(\Gamma') \geq
(\frac{k}{k + 1} + \frac{\alpha}{2})n$ and $|V_0| \leq \eps n$ (see
Theorem~\ref{thm:degree-form-regularity}\ref{reg-except} and \ref{reg-degree}),
by using Chernoff's inequality and the union bound, we may assume that
\[
  \deg_{\Gamma'}(v, W) \geq \Big( \frac{k}{k + 1} + \frac{\alpha}{4} \Big)|W|
  \qquad \text{and} \qquad \deg_{\Gamma'}(v, X) \geq \Big( \frac{k}{k + 1} +
  \frac{\alpha}{4} \Big)|X|,
\]
for all $v \in V(G)$, as $\eps \leq \alpha/8$.

By the Absorbing-Covering Lemma (Lemma~\ref{lem:absorbing-lemma}) applied with
$X$ and $W$ there is an $X$-absorbing $(2k + 1)$-path $P_1$ connecting some
$\mbs_1$ to $\mbt_1$. Note that
Lemma~\ref{lem:absorbing-lemma}\ref{A-cliques-s}--\ref{A-cliques-t} also
provides some $z \in \bigcap_{i \in [k + 1]} N_R(i)$, such that $\rev(\mbs_1)$
is $(\mbW_s, \rho)$-extendible and $\mbt_1$ is $(\mbW_t, \rho)$-extendible,
where
\[
  \mbW_s = (W_z, W_{k + 1}, \dotsc, W_2) \setminus V(P_1) \quad \text{and} \quad
  \mbW_t = (W_z, W_1, \dotsc, W_k) \setminus V(P_1).
\]

Let $V_0' := V_0 \cup (W \setminus V(P_1))$ be the set of all the vertices from
the exceptional set of the regular partition $\cV$ and all the unused vertices
of $W$. Note that, by the regularity lemma
(Theorem~\ref{thm:degree-form-regularity}\ref{reg-except}) and the
Absorbing-Covering Lemma (Lemma~\ref{lem:absorbing-lemma}\ref{A-leftover}),
we have $|V_0'| \leq \eps n + 1.1\gamma n \leq 2\gamma n$. Crucially, $V_0'$ is
significantly smaller than $X$. We apply Lemma~\ref{lem:absorbing-lemma} once
again, this time with $V_0'$ (as $X$) and $X$ (as $W$) to obtain a $(2k +
1)$-path $P_2$ connecting some $\mbs_2$ to $\mbt_2$ which is $V_0'$-absorbing.

\begin{assumptions}[Lemma~\ref{lem:absorbing-lemma}]
  We have $|X| \geq (\xi/2)(1 - \eps)n \geq \mu n$, $|V_0'| \leq 2\gamma n \leq
  \xi|X|$, and $\deg_{\Gamma'}(v, X) \geq (\frac{k}{k + 1} +
  \frac{\alpha}{4})|X|$, for all $v \in V(G)$.
\end{assumptions}

By Lemma~\ref{lem:absorbing-lemma}\ref{A-cliques-s}--\ref{A-cliques-t} there
exists $z' \in \bigcap_{i \in [k + 1]} N_R(i)$ such that $\rev(\mbs_2)$ is
$(\mbX_s, \rho)$-extendible and $\mbt_2$ is $(\mbX_t, \rho)$-extendible, where
$\mbX_s = (X_{z'}, X_{k + 1}, \dotsc, X_2) \setminus V(P_2)$ and $\mbX_t =
(X_{z'}, X_1, \dotsc, X_k) \setminus V(P_2)$.

For such $z$ and $z'$, by the definition of a reduced graph and
Proposition~\ref{prop:large-reg-inheritance} we have that
\begin{equation}\label{eq:R-tuples}
  \begin{split}
    & (W_z, W_1, \dotsc, W_k), \quad (W_2, \dotsc, W_{k + 1}, X_1), \quad (X_2,
    \dotsc, X_{k + 1}, X_{z'}), \\
    & (X_{z'}, X_1, \dotsc, X_k), \quad (X_2, \dotsc, X_{k + 1}, W_1), \quad (W_2,
    \dotsc, W_{k + 1}, W_z)
  \end{split}
\end{equation}
are $(\eps', d/2, \Gamma')$-regular $(k + 1)$-tuples. We now apply
Corollary~\ref{cor:connecting-lemma} with
\begin{align*}
  \big\{ \big( \mbW_t, W_{k + 1} \setminus V(P_1), X_1 \setminus V(P_2),
  \rev(\mbX_s) \big),
  \big( \mbX_t, X_{k + 1} \setminus V(P_2), W_1 \setminus V(P_1), \rev(\mbW_s)
  \big) \big\}
\end{align*}
(as $\{\mbV_1, \mbV_2\}$), and $\{(\mbt_1, \mbs_2), (\mbt_2, \mbs_1)\}$, to get
two $(2k + 1)$-paths: $P_{12}$ connecting $\mbt_1$ to $\mbs_2$ and $P_{21}$
connecting $\mbt_2$ to $\mbs_1$ (see Figure~\ref{fig:closing-the-cycle}).

\begin{figure}[!htbp]
  \centering
  \includegraphics[scale=0.75]{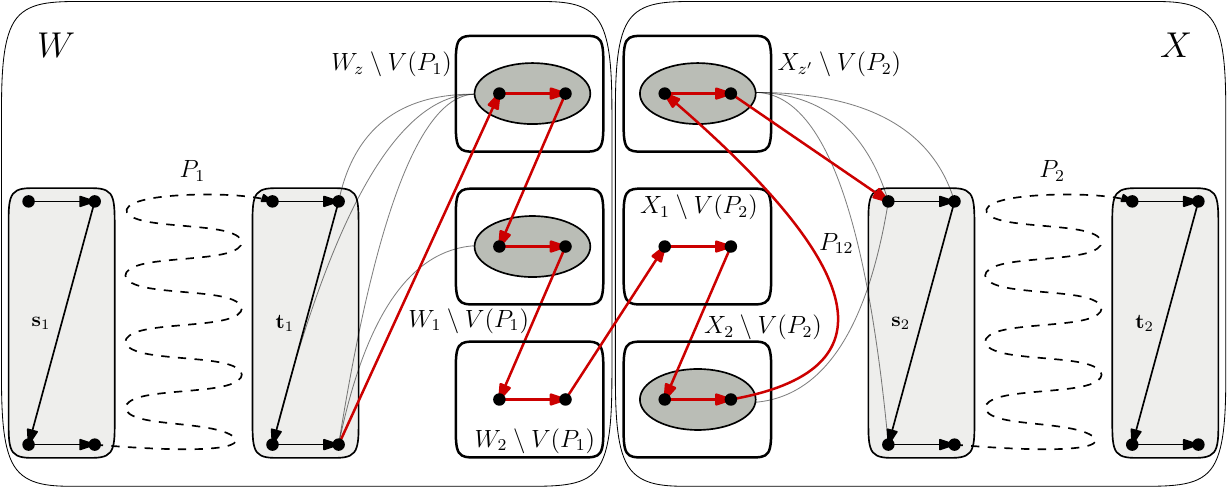}
  \caption{An example of connecting the $(2k + 1)$-path $P_1$ to the $(2k +
  1)$-path $P_2$ by a $(2k + 1)$-path $P_{12}$ connecting $\mbt_1$ to $\mbs_2$
  (given in red) for $k = 1$. For simplicity, we depict only skeletons of the
  mentioned $(2k + 1)$-paths.}
  \label{fig:closing-the-cycle}
\end{figure}

\begin{assumptions}[Corollary~\ref{cor:connecting-lemma}]
  Note that
  \[
    |W_i \setminus V(P_1)| \geq 0.9\gamma|W_i| \geq
    (\xi^3/16)\tilde n \geq \mu n \quad \text{and} \quad |X_i \setminus V(P_2)|
    \geq 0.9\gamma|X_i| \geq (\xi^3/16)\tilde n \geq \mu n,
  \]
  by Lemma~\ref{lem:absorbing-lemma}\ref{A-leftover}. Hence, due to
  \eqref{eq:R-tuples} and as $2\eps'/\gamma \leq \eps''$, the required sets are
  $(\eps'', d/4, \Gamma')$-regular by
  Proposition~\ref{prop:large-reg-inheritance}. Lastly, by
  Lemma~\ref{lem:absorbing-lemma}\ref{A-cliques-s}--\ref{A-cliques-t} the
  cliques $\mbt_1$, $\rev(\mbs_2)$, $\mbt_2$, and $\rev(\mbs_1)$, are extendible
  with respect to $\rho$ and appropriate sets used as $\mbV_1, \mbV_2$.
\end{assumptions}

Clearly, merging together the paths $P_1$, $P_{12}$, $P_2$, and $P_{21}$, closes
a $(2k + 1)$-cycle $C$. Finally, let $Q_1 \subseteq X$ and $Q_2 \subseteq V_0'$
be the subsets of vertices which do not belong to $C$. For $i \in \{1, 2\}$,
exchanging $P_i$ by a $(2k + 1)$-path $P_i^\star$ with the same endpoints as
$P_i$ and such that $V(P_i^\star) = V(P_i) \cup Q_i$, we finally obtain the $(2k
+ 1)$-st power of a Hamilton cycle. \qed

\paragraph{Note.} While finishing the present manuscript we learned that
Antoniuk, Dudek, Reiher, Ruci\'{n}ski, and Schacht have independently obtained
similar results.

{\small \bibliographystyle{abbrv} \bibliography{references}}

\end{document}